\documentclass[11pt]{article}

\usepackage[utf8]{inputenc}

\usepackage[margin=1in]{geometry}
\usepackage{graphicx}
\usepackage{graphics}
\usepackage{subcaption}

\usepackage{amsmath,amsfonts,amssymb}
\usepackage{amsthm}
\usepackage{mathrsfs}
\usepackage{esint}
\usepackage{mathtools}
\usepackage{bbm}
\usepackage{scalerel}

\let\olddiamond\diamond
\let\oldsquare\square % Must go before mathabx

\usepackage{mathabx}

\renewcommand{\square}{\oldsquare}
\renewcommand{\diamond}{\olddiamond}

\usepackage{xfrac}
\usepackage{nicefrac}

\usepackage{booktabs}
\usepackage{array}

\usepackage{paralist}
\usepackage{enumerate}
\usepackage{enumitem}

\usepackage{verbatim}
\usepackage{fancyvrb}
\usepackage[normalem]{ulem}
\usepackage{cancel}
\usepackage{accents}

\usepackage[usenames,dvipsnames]{xcolor}
\usepackage[colorlinks=true, pdfstartview=FitV, linkcolor=blue, citecolor=blue, urlcolor=blue]{hyperref}

\usepackage[capitalize]{cleveref}
\crefname{enumi}{}{}
\crefname{equation}{}{}

\usepackage{tikz}
\usetikzlibrary{calc}
\usepackage{pgf}

\usepackage[nottoc,notlot,notlof]{tocbibind}

\usepackage{titlesec}

\newsavebox{\precone}
\savebox{\precone}{%
  \begin{tikzpicture}[baseline=-0.65ex, scale=0.13,inner sep=0pt, outer sep=0pt]
    \draw (0,0) ellipse (1 and 0.35);
    \draw (-1,0) -- (0,-2) -- (1,0);
    \draw[dashed] (-1,0) arc (180:360:1 and 0.35);
  \end{tikzpicture}
}

\newcommand{\cone}{\makebox[8pt][l]{\usebox{\precone}}}
\newcommand{\bacon}{\mathbin{\scalerel*{\cone}{\bigtriangleup}}}

\newcommand{\baconsol}{\bacon_{\scriptscriptstyle{\mathrm{sol}}}}
\newcommand{\baconsub}{\bacon_{\scriptscriptstyle{\mathrm{sub}}}}
\newcommand{\baconconv}{\bacon_{\scriptscriptstyle{\mathrm{conv}}}}
\newcommand{\bacontrunc}{\bacon_{\scriptscriptstyle{\mathrm{trunc}}}}
\newcommand{\baconpow}{\bacon_{\scriptscriptstyle{\mathrm{pow}}}}
\newcommand{\baconmax}{\bacon_{\scriptscriptstyle{\mathrm{max}}}}

\numberwithin{equation}{section}
\numberwithin{figure}{section}

\newtheorem{theorem}{Theorem}[section]

\newtheorem{corollary}[theorem]{Corollary}
\newtheorem{proposition}[theorem]{Proposition}
\newtheorem{lemma}[theorem]{Lemma}

\theoremstyle{definition}
\newtheorem{definition}[theorem]{Definition}

\newtheorem{remark}[theorem]{Remark}

\DeclareMathOperator{\dist}{dist}

\newcommand*{\supp}{\ensuremath{\mathrm{supp\,}}}

\newcommand{\sgn}{\operatorname{sgn}}

\newcommand*{\N}{\ensuremath{\mathbb{N}}}

\newcommand*{\Z}{\ensuremath{\mathbb{Z}}}

\newcommand*{\R}{\ensuremath{\mathbb{R}}}

\newcommand{\E}{\mathbb{E}}

\newcommand*{\Zd}{\ensuremath{\mathbb{Z}^d}}
\newcommand*{\Rd}{\ensuremath{\mathbb{R}^d}}

\newcommand{\A}{\mathcal{A}}

\renewcommand{\b}{\ensuremath{\mathbf{b}}}

\newcommand{\f}{\mathbf{f}}

\renewcommand{\a}{\mathbf{a}}

\newcommand*{\Id}{\ensuremath{\mathrm{I}_d}}

\newcommand{\eps}{\varepsilon}
\renewcommand*{\tilde}{\widetilde}

\DeclareSymbolFont{boldoperators}{OT1}{cmr}{bx}{n}
\SetSymbolFont{boldoperators}{bold}{OT1}{cmr}{bx}{n}

\newcommand{\ep}{\eps}

\makeatletter 
\newcommand{\negphantom}{\v@true\h@true\negph@nt} 
\newcommand{\neghphantom}{\v@false\h@true\negph@nt} 
\newcommand{\negph@nt}{\ifmmode\expandafter\mathpalette 
  \expandafter\mathnegph@nt\else\expandafter\makenegph@nt\fi} 
\newcommand{\makenegph@nt}[1]{% 
  \setbox\z@\hbox{\color@begingroup#1\color@endgroup}\finnegph@nt} 
\newcommand{\finnegph@nt}{% 
  \setbox\tw@\null 
  \ifv@ \ht\tw@\ht\z@\dp\tw@\dp\z@\fi \ifh@\wd\tw@-\wd\z@\fi\box\tw@} 
\newcommand{\mathnegph@nt}[2]{% 
  \setbox\z@\hbox{$\m@th #1{#2}$}\finnegph@nt} 
\makeatother

\newcommand{\Wminusul}[2]{\hat{\phantom{W}}\negphantom{W}\underline{\phantom{H}}\negphantom{H}W^{#1,#2}}
\newcommand{\Wminusnoul}[2]{{\phantom{W}}\negphantom{W}\Wul{#1}{#2}}
\newcommand{\Wul}[2]{\underline{\phantom{H}}\negphantom{H}W^{#1,#2}}
\newcommand{\Woul}[2]{\underline{\phantom{H}}\negphantom{H}W_0^{#1,#2}}

\newcommand{\Besov}[3]{\underline{B}_{#1,#2}^{#3}}

\newcommand{\cs}{\mathfrak{c}}

\newcommand{\BesovDualSum}[3]{%
  \mathord{\mathring{\mkern-2mu\underline{B}}}^{#3}_{#1,#2}%
}
\newcommand{\BesovDual}[3]{%
  \mathord{\widehat{\mkern-2mu\underline{B}}}\!^{#3}_{#1,#2}%
}

\def\Xint#1{\mathchoice
{\XXint\displaystyle\textstyle{#1}}%
{\XXint\textstyle\scriptstyle{#1}}%
{\XXint\scriptstyle\scriptscriptstyle{#1}}%
{\XXint\scriptscriptstyle\scriptscriptstyle{#1}}%
\!\int}
\def\XXint#1#2#3{{\setbox0=\hbox{$#1{#2#3}{\int}$}
\vcenter{\hbox{$#2#3$}}\kern-.5\wd0}}

\def\fint{\Xint-}

\newcommand{\avsum}{\mathop{\mathpalette\avsuminner\relax}\displaylimits}

\makeatletter
\newcommand\avsuminner[2]{%
  {\sbox0{$\m@th#1\sum$}%
   \vphantom{\usebox0}%
   \ooalign{%
     \hidewidth
     \smash{\,\rule[.23em]{8.8pt}{1.1pt} \relax}%
     \hidewidth\cr
     $\m@th#1\sum$\cr
   }%
  }%
}

\newcommand\avsuminnerr[2]{%
  {\sbox0{$\m@th#1\sum$}%
   \vphantom{\usebox0}%
   \ooalign{%
     \hidewidth
     \smash{\,\rule[.23em]{6pt}{0.7pt} \relax}%
     \hidewidth\cr
     $\m@th#1\sum$\cr
   }%
  }%
}
\makeatother

\let\originalleft\left
\let\originalright\right
\renewcommand{\left}{\mathopen{}\mathclose\bgroup\originalleft}
\renewcommand{\right}{\aftergroup\egroup\originalright}

\newcommand{\cu}{\square}

\newcommand{\indc}{{\mathbbm{1}}}
\renewcommand{\hat}{\widehat}

\newcommand{\addperiod}[1]{#1.}
\titleformat{\section}
   {\centering\normalfont\Large}{\thesection.}{0.5em}{}
\titleformat{\subsection}[runin]
  {\normalfont\bfseries}
  {\thesubsection.}
  {0.5em}
  {\addperiod}
\titleformat{\subsubsection}[runin]
  {\normalfont\bfseries}
  {\thesubsubsection.}
  {0.5em}
  {\addperiod}
\titleformat*{\subsubsection}{\normalfont\itshape}
\titleformat*{\paragraph}{\bfseries}
\titleformat*{\subparagraph}{\large\bfseries}

\title{Coarse-grained ellipticity and De Giorgi-Nash-Moser theory}

\author{S. Armstrong
\thanks{CNRS \& Laboratoire Jacques-Louis Lions, Sorbonne Universit\'e 
and
Courant Institute of Mathematical Sciences, New York University.
{\footnotesize \href{mailto:scottnarmstrong@gmail.com}{scottnarmstrong@gmail.com}.}}
\and
B. Avelin
\thanks{Department of Mathematics, Uppsala University.
{\footnotesize \href{mailto:benny.avelin@math.uu.se}{benny.avelin@math.uu.se}.}}
\and
C. De Filippis
\thanks{Dipartimento SMFI, Universitá di Parma.
{\footnotesize \href{mailto:cristiana.defilippis@unipr.it}{cristiana.defilippis@unipr.it}.}}
\and
T. Kuusi
\thanks{Department of Mathematics and Statistics, University of Helsinki.
{\footnotesize \href{mailto:tuomo.kuusi@helsinki.fi}{tuomo.kuusi@helsinki.fi}.}}
\and
G. Mingione
\thanks{Dipartimento SMFI, Universitá di Parma.
{\footnotesize \href{mailto:giuseppe.mingione@unipr.it}{giuseppe.mingione@unipr.it}.}}
}

\date{January 9, 2026 }

\begin{document}

\maketitle

\begin{abstract}
  We prove local boundedness and a Harnack inequality for nonnegative weak solutions of the equation $-\nabla\cdot(\a(x)\nabla u)=0$ under a coarse-grained ellipticity assumption on the symmetric coefficient field~$\a$. Coarse-grained ellipticity is a scale-dependent condition, defined for fields with only $\a,\a^{-1}\in L^1$, in terms of families of effective diffusion matrices on triadic cubes of all sizes, and our estimates depend quantitatively on a corresponding coarse-grained ellipticity ratio. We show that coarse-grained ellipticity can be enforced by purely negative Sobolev regularity hypotheses: if $\a\in L^1\cap W^{-s,p}(U)$ and $\a^{-1}\in L^1\cap W^{-t,q}(U)$ for exponents $p,q\in[1,\infty]$ and $s,t\in[0,1)$ satisfying $s<1-\nicefrac1p$, $t<1-\nicefrac1q$ and
  \begin{equation*}
    \frac{s+t}{2} + \frac{d}{2}\Bigl(\frac1p+\frac1q\Bigr) < 1,
  \end{equation*}
  then $\a$ is coarse-grained elliptic in~$U$ and every nonnegative solution satisfies a quantitative unit-scale Harnack inequality. In particular, when $s=t=0$ we recover Trudinger's classical result under the integrability condition $\a\in L^p$, $\a^{-1}\in L^q$ with $1/p+1/q<2/d$, and we obtain the sharp scaling of the Harnack constant in terms of $\|\a\|_{L^p}$ and $\|\a^{-1}\|_{L^q}$. More importantly, our criteria apply to new classes of degenerate and singular coefficient fields for which $\a,\a^{-1}\notin L^{1+\delta}$ for all $\delta>0$, including examples generated by singular fractal measures and Gaussian multiplicative chaos, beyond the reach of previous approaches based solely on integrability assumptions.
\end{abstract}

\tableofcontents

\section{Introduction}

\subsection{Main results}
We study regularity properties of solutions of linear elliptic equations
\begin{equation}
  -\nabla \cdot \a \nabla u = 0
  \qquad \text{in} \ U \subseteq \Rd
  \,,
  \label{e.pde}
\end{equation}
under the assumption that the symmetric coefficient field~$\a(x)$ is \emph{coarse-grained elliptic}. This coarse-graining framework associates, to each coefficient field, a family of scale-dependent, effective diffusion matrices and ellipticity constants which encode the behavior of the solutions (and subsolutions) at different scales. This new notion of ellipticity was introduced in~\cite{AK.HC} as a tool for analyzing large-scale properties of equations with random coefficients. Here we show that a suitable modification of the same structure controls the deterministic behavior of solutions.

\smallskip

We present two quantitative, scale-dependent estimates: a local boundedness principle for weak subsolutions of~\eqref{e.pde}, and a Harnack inequality for nonnegative weak solutions. These results generalize, for symmetric coefficient fields, the classical results of De Giorgi~\cite{DeGiorgi}, Nash~\cite{Nash} and Moser~\cite{Moser1,Moser2} in the uniformly elliptic setting.

\smallskip

The definitions of coarse-grained ellipticity (Definition~\ref{def.cg.ellipticity}), the weighted Sobolev space~$H^1_\a$ (Section~\ref{ss.sobolev.spaces}), and the coarse-grained ellipticity ratio~$\Theta_{s,t}(\cu_0\,;\baconsub)$ (see~\eqref{e.cg.ellipticity.ratio}) can be found in Section~\ref{s.prelims}, below.

\begin{theorem}[Local boundedness principle]
  \label{t.local.boundedness.intro}
  Suppose that~$\a$ is a coarse-grained elliptic coefficient field on the unit cube~$\cu_0 \coloneqq (-\frac12,\frac12)^d$.
  Assume~$s,t \in (0,1)$ with $s+t < 1$. Denote~$\sigma \coloneqq 1-s-t>0$.
  There exists a constant $C(s,t,d)<\infty$ such that, for every~$u\in H^1_\a(\cu_0)$ satisfying
  \begin{equation}
    -\nabla \cdot \a \nabla u \leq  0 \quad \mbox{in} \ \cu_0
    \,,
  \end{equation}
  we have the estimate
  \begin{equation}
    \label{e.local.boundedness.intro}
    \sup_{\frac{1}{2}\cu_0} u
    \leq
    C \Theta_{s,t}^{\nicefrac{d}{4 \sigma}} (\cu_0\,;\baconsub)
    \|u_+\|_{L^2(\cu_0)}\,.
  \end{equation}
\end{theorem}

\begin{theorem}[Harnack inequality]
  \label{t.Harnack.intro}
  Suppose that~$\a$ is a coarse-grained elliptic coefficient field on~$\cu_0$. Assume~$s,t \in (0,1)$ with~$s+t<1$. There exists a constant~$C = C(s,t,d) <\infty$ such that every nonnegative weak solution~$u\in H^1_\a(\cu_0)$ of
  \begin{equation}
    -\nabla \cdot \a \nabla u = 0 \quad \mbox{in} \ \cu_0
    \,,
  \end{equation}
  satisfies the estimate
  \begin{equation}
    \label{e.Harnack.intro}
    \sup_{\frac{1}{8}\cu_0} u
    \leq
    \exp\bigl(C\Theta_{s,t}^{\nicefrac{1}{2}}(\cu_0\,;\baconsub)
    \bigr) \inf_{\frac{1}{8}\cu_0} u
    \,.
  \end{equation}
\end{theorem}

Observe that the constant in~\eqref{e.Harnack.intro} is explicit in its dependence on the coarse-grained ellipticity ratio of the field~$\a$. Specializing to uniformly elliptic coefficients,~$\Theta_{s,t}(\cu_0\,;\baconsub)$ is bounded from above by the usual ratio of uniform ellipticity constants and thus~\eqref{e.Harnack.intro} reduces to the classical estimate of Bombieri and Giusti~\cite{BG} which is known to be optimal (see Remark~\ref{rem.optimal}).

\smallskip

The coarse-grained  matrices and ellipticity constants are defined for symmetric, positive semi-definite, coefficient fields~$\a$ satisfying merely~$\a\in L^1$ and~$\a^{-1} \in L^1$. Such a field is \emph{coarse-grained elliptic} if its coarse-grained ellipticity constants are positive and finite. A simple sufficient condition for coarse-grained ellipticity, and thus for the validity of our Harnack inequality, can be formulated in terms of the \emph{negative Sobolev regularity} of the coefficient field.
As we will show, for every~$p,q\in [1,\infty]$ and~$s,t\in [0,1)$ satisfying $s < 1 - \nicefrac1p$,~$t < 1 - \nicefrac1q$ and
\begin{equation}
  \frac{s+t}{2} + \frac{d}{2}\biggl(\frac{1}{p}+\frac{1}{q}\biggr ) < 1
  \label{e.Sobolev.exponent.condition}
\end{equation}
we have the implication
\begin{equation}
  \a \in (L^1 \cap W^{-s,p}) (U)
  \  \ \&  \  \
  \a^{-1} \in (L^1 \cap W^{-t,q}) (U)
  \  \implies \
  \text{ $\a$ is coarse-grained elliptic in $U$\,,}
  \label{e.Sobolev.ellipticity.condition}
\end{equation}
Here~$W^{-s,p}(U)$ denotes the dual space of the fractional Sobolev space~$W^{s,p'}(U)$ (see Section~\ref{ss.notation}). The implication~\eqref{e.Sobolev.ellipticity.condition} is also quantitative: the coarse-grained ellipticity constants can be controlled explicitly by the negative Sobolev norms of the field. As far as we are aware, Theorem~\ref{t.Harnack.intro} is the first instance in which a \emph{negative} regularity assumption on the coefficient field is used as a structural hypothesis to obtain De Giorgi-Nash-Moser type estimates.

\smallskip

Specializing~\eqref{e.Sobolev.exponent.condition},~\eqref{e.Sobolev.ellipticity.condition} to~$s=t=0$, the classical integrability condition emerges:
\begin{equation}
  \a \in L^p \ \ \& \ \
  \a^{-1}\in L^q
  \quad \mbox{with} \quad \frac1p + \frac1q < \frac2d
  \  \implies \
  \text{ $\a$ is coarse-grained elliptic in $U$.}
  \label{e.classical.integrability}
\end{equation}
Therefore Theorem~\ref{t.Harnack.intro} yields, as a special case, a large-scale Harnack inequality for coefficient fields satisfying~\eqref{e.classical.integrability}, recovering the classical result of Trudinger~\cite{Trudinger}. Moreover, we obtain the sharp scaling of the Harnack constant in terms of~$\| \a \|_{L^p(\cu_0)}$ and~$\| \a^{-1} \|_{L^q(\cu_0)}$, given by
\begin{equation}
  \exp\bigl( C \| \a \|_{L^p(\cu_0)}^{\nicefrac12}  \| \a^{-1} \|_{L^q(\cu_0)}^{\nicefrac12} \bigr)
  \,,
  \label{e.Harnack.constant.ss}
\end{equation}
where~$C(d,p,q)<\infty$. This scaling of the constant is optimal (see Remark~\ref{rem.optimal}) and appears to be new in the setting of~\eqref{e.classical.integrability}, improving on the exponential dependence in~$\| \a \|_{L^p(\cu_0)} \| \a^{-1} \|_{L^q(\cu_0)}$ obtained in~\cite{BS}.

\smallskip

The range of integrability exponents~$p$ and~$q$ covered by~\eqref{e.classical.integrability} is not optimal. The sharp result, proved by Bella and Sch\"affner~\cite{BS}, asserts that local boundedness and Harnack inequalities are valid under the strictly weaker condition
\begin{equation}
  \a \in L^p \quad \& \quad
  \a^{-1}\in L^q \,,
  \qquad \frac1p + \frac1q < \frac2{d-1}
  \,.
  \label{e.BellaSchaffner}
\end{equation}
The condition~\eqref{e.BellaSchaffner} is known to be optimal in dimensions~$d\geq 4$, as was demonstrated by explicit example in~\cite{FSS}.
Our approach does not recover this full range of integrability exponents: if we specialize Theorem~\ref{t.Harnack.intro} to fields satisfying~$\a\in L^p(\cu_0)$ and~$\a^{-1} \in L^q(\cu_0)$ then its applicability is limited to~$p$ and~$q$ satisfying~\eqref{e.classical.integrability}.
Within the class of coefficient fields satisfying~$\a\in L^p$ and~$\a^{-1}\in L^q$, the results of~\cite{BS} are therefore stronger.

\smallskip

The contribution of the present paper lies in a different direction. The negative Sobolev regularity assumption in~\eqref{e.Sobolev.ellipticity.condition} allows us to treat coefficient fields that are substantially more singular or degenerate than those covered by an integrability condition such as~\eqref{e.classical.integrability}. Indeed, one can construct coefficient fields satisfying~\eqref{e.Sobolev.ellipticity.condition} for which~$\a,\a^{-1}\notin L^{1+\delta}$ for every~$\delta>0$; see Appendix~\ref{app.examples} for several explicit examples. The key point is that the large-scale regularity of solutions is governed by coarse-grained ellipticity, and the latter can be controlled in terms of \emph{scale-discounted spatial averages} of~$\a$ and~$\a^{-1}$, rather than in terms of strong norms of these fields.

\smallskip

Like the results of Trudinger~\cite{Trudinger} and Bella--Sch\"affner~\cite{BS}, our estimates are intrinsically at unit scale. The coarse-grained ellipticity ratios associated with smaller cubes may grow like a power of the inverse of the size, and thus the Harnack constant may deteriorate rapidly on smaller scales. In particular, an iteration of the estimate does not, in general, yield a continuity estimate. This is in contrast with degenerate ellipticity hypotheses based on Muckenhoupt weights, such as the $A_2$ condition, under which one imposes a scale-invariant weighted Poincar\'e inequality and obtains H\"older continuity and scale-invariant Harnack inequalities for solutions; see, for example, Fabes, Kenig and Serapioni~\cite{FKS}. In some applications, this is not an obstacle, since large-scale regularity---rather than pointwise regularity at all scales---is the relevant notion; the quintessential example being homogenization theory. Indeed, coarse-grained ellipticity provides a natural structural condition in that context, as has been already demonstrated in~\cite{AK.HC} and a very recent work of Lau~\cite{Lau} proves a general abstract, qualitative homogenization statement under a coarse-grained ellipticity assumption.

\subsection{Overview of the proofs}

In Section~\ref{s.prelims} we associate to each triadic cube~$\cu \subseteq\cu_0$ and certain (possibly nonconvex) cones~$\bacon(\cu) \subseteq H^1_\a(\cu)$ a pair of
effective diffusion operators~$\a(\cu;\bacon)$ and~$\a_*(\cu)$ obtained from variational formulas over~$\bacon$. Here~$\a_*(\cu)$ is the usual coarse-grained matrix used in~\cite{AKM.Book,AK.HC}, and~$\a(\cu;\bacon)$ is a generalization of the coarse-grained matrix~$\a(\cu)$, coinciding with it only in the case~$\bacon$ is the linear space of solutions.
These operators quantify how the energy~$\int_\cu \nabla u \cdot \a\nabla u$ controls the spatially-averaged gradients and fluxes of elements of~$\bacon$ on~$\cu$. The coarse-grained ellipticity constants~$\Lambda_{s}(\cu_0 \,; \bacon)$ and~$\lambda_{s}(\cu_0)$ are defined in~\eqref{eq:besov-norms} by combining~$\|\a(\cu)\|$ and $|\a_*(\cu)^{-1}|$ over all triadic subcubes~$\cu\subseteq\cu_0$ with scale-discounted weights (the exponent~$s\in [0,1)$ determines the strength of the discount). Our arguments use only the finiteness of these quantities. The coarse-grained ellipticity ratio~$\Theta_{s,t}(\cu_0\,;\bacon)$ is defined in~\eqref{e.cg.ellipticity.ratio} as the ratio of~$\Lambda_{s}(\cu_0 \,; \bacon)$ and~$\lambda_{s}(\cu_0)$.

\smallskip

The main analytic input of the paper is a collection of functional and elliptic inequalities which depend on the field~$\a$ only via the coarse-grained ellipticity constants. In Section~\ref{ss.functional} we prove a \emph{coarse-grained Sobolev--Poincar\'e} inequality at unit scale: for $u\in H^1_\a(\cu_0)$, the energy~$\int_{\cu_0}\a\nabla u\cdot\nabla u$ controls a fractional Sobolev norm (or an $L^{2^*}$-norm)
of $u-(u)_{\cu_0}$, with a constant depending explicitly on~$\lambda_{t}(\cu_0)$. This is written as an estimate of weak norms of gradients of~$u$, and we also obtain analogous estimates on weak norms of fluxes of elements of~$\bacon$ in terms of~$\Lambda_{s}(\cu_0\,;\bacon)$.  In Section~\ref{ss.cg.caccioppoli}, we establish a \emph{coarse-grained Caccioppoli inequality} for convex transforms $\Psi(v)$ of~$v\in \bacon$, for cones~$\bacon$ which allow~$v$ to be one~$(u-k)_+$,~$u^{p/2}$, and~$\log u$,  for a solution~$u$. Combining these two ingredients yields reverse H\"older inequalities for truncations and powers of solutions, and a logarithmic Caccioppoli estimate for~$\log u$, all with constants expressed explicitly in terms of the coarse-grained ellipticity ratio. This provides a complete coarse-grained version of the De Giorgi--Moser toolkit.

\smallskip

Once these coarse-grained inequalities are available, the proofs of Theorems~\ref{t.local.boundedness.intro} and~\ref{t.Harnack.intro} follow the classical De Giorgi, Moser and Bombieri--Giusti schemes. The reverse H\"older inequality for truncations yields, via a standard De Giorgi iteration, the local boundedness estimate~\eqref{e.local.boundedness.intro} for subsolutions. The reverse H\"older inequality for powers is used to perform Moser iterations for positive and negative exponents, giving~$L^p$-to-$L^\infty$ bounds for solutions from above and below. The logarithmic estimate, together with a Bombieri--Giusti crossover lemma, compares positive and negative moments of $u$ and leads to the unit-scale Harnack inequality~\eqref{e.Harnack.intro} with the constant~$\exp\big(C\Theta_{s,t}^{\nicefrac12}\big)$. The structure of the argument is thus entirely classical---what is new is that every step relies only on  coarse-grained ellipticity.

\subsection{Examples}

In Appendix~\ref{app.examples} we collect several examples of coefficient fields~$\a(x)$ satisfying the negative Sobolev assumption~\eqref{e.Sobolev.ellipticity.condition} but no~$L^{1+\delta}$ bound for either~$\a$ or~$\a^{-1}$. These examples are not pathological. One class is obtained by regularizing singular measures~$\mu_F$ supported on Ahlfors regular sets~$F$ of Hausdorff dimension belonging to the interval~$(d-2,d)$ by setting~$\mu_F^\delta = \mu_F \ast \eta_\delta$, where~$\{ \eta_\delta\}$ is the standard mollifier.  Setting~$\a_\delta = 1 + \mu_F^\delta$ then gives us a coefficient field satisfying our assumptions (see Appendix~\ref{ss.singular.measure}). Such fields arise naturally as effective densities or conductivities concentrated on lower-dimensional structures.

\smallskip

A second example, presented in Appendix~\ref{ss.GMC}, is provided by a suitable modification of Gaussian multiplicative chaos, where we take a log-correlated Gaussian field, construct the associated multiplicative chaos measure, mollify at scale~$\delta$ and add a constant. These log-normal weights are standard in models of turbulent transport and random media. For both families the~$L^p$-norms diverge for every~$p>1$, whereas the negative Sobolev and coarse-grained ellipticity bounds remain uniformly controlled, so the estimates of Theorems~\ref{t.local.boundedness.intro} and~\ref{t.Harnack.intro} apply but previous results based on the integrability conditions~\eqref{e.classical.integrability} or~\eqref{e.BellaSchaffner} do not.

\section{Coarse-grained ellipticity}
\label{s.prelims}

In this section we introduce our notation and define coarse-grained ellipticity.

\subsection{Notation}\label{ss.notation}
We work throughout in spatial dimension~$d \geq 2$. For a measurable set
$U \subseteq \Rd$ with positive Lebesgue measure~$|U|>0$, we denote its
average integral and volume-normalized $L^p$-norm by
\begin{equation*}
  (f)_U \coloneqq \fint_U f(x) \, dx \coloneqq \frac{1}{|U|} \int_U f(x) \, dx,
  \qquad
  \|f\|_{\underline{L}^p(U)} \coloneqq
  \biggl(
  \fint_U |f(x)|^p \, dx
  \biggr)^{\!\nicefrac{1}{p}},
  \quad p \in [1,\infty).
\end{equation*}
Given a finite index set $I$, we write for the average of a family of scalars $\{f_i\}_{i\in I}$
\begin{equation*}
  \avsum_{i\in I} f_i \coloneqq
  \frac{1}{|I|} \sum_{i\in I} f_i\,.
\end{equation*}

For $s \in (0,1)$ and $p \in [1,\infty)$, we use the volume-normalized
fractional Sobolev space $\Wul{s}{p}(U)$ consisting of all
$f \in L^p(U)$ such that
\begin{equation*}
  [f]_{\Wul{s}{p}(U)} \coloneqq
  \biggl(
  \fint_U \int_U \frac{|f(x)-f(y)|^p}{|x-y|^{d+sp}} \, dx \, dy
  \biggr)^{\!\nicefrac{1}{p}} < \infty,
\end{equation*}
endowed with the norm
\begin{equation*}
  \|f\|_{\Wul{s}{p}(U)} \coloneqq
  \bigl(
  |U|^{-\frac{sp}{d}}\|f\|_{\underline{L}^p(U)}^p
  + [f]_{\Wul{s}{p}(U)}^p
  \bigr)^{\!\nicefrac{1}{p}}.
\end{equation*}
Its dual, in the volume-normalized convention, is
written $\Wminusul{-s}{p'}(U)$, where
$p' \coloneqq \frac{p}{p-1}$ is the H\"older conjugate of~$p$.
The usual (non-volume-normalized) fractional Sobolev space is denoted
by $W^{s,p}(U)$.

We also define the space $\Woul{s}{p}(U)$ as the closure of $C^\infty_c(U)$ in $\Wul{s}{p}(U)$ with respect to the norm $\|\cdot\|_{\Wul{s}{p}(U)}$ and its dual $\Wminusnoul{-s}{p'}(U)$.

We work with triadic cubes. For each $m \in \Z$ we set
\begin{equation*}
  \cu_m \coloneqq
  \bigl( -\tfrac12 3^m , \tfrac12 3^m \bigr)^{\!d},
\end{equation*}
and we let $\Z^d$ be the standard integer lattice in~$\Rd$.

Matrices are always denoted by boldface letters such as~$\a$ and are
assumed symmetric. For a symmetric matrix~$\a$ we write $|\a|$ for its
spectral norm, $\a^{-1}$ for its inverse (when it exists), and
$\a^{\nicefrac12}$ for its principal square root. When we write $\a \in L(U)$ where $L$ is a space of real valued functions, we mean that each entry of the matrix~$\a(x)$ belongs to~$L(U)$. For an operator
$T : \Rd \to \R$ we denote by~$\|T\|$ its operator norm. 

For a function $u$ we denote its positive part by
$u_+ \coloneqq \max\{u,0\}$. We also use the shorthand
\begin{equation*}
  \cs_s \coloneqq 1 - 3^{-s},
  \qquad s \in [0,1),
\end{equation*}
so that $\cs_s \sum_{n=0}^\infty 3^{-sn} = 1$ for every $s \in (0,1]$.

Throughout the paper, the symbol $C$ denotes a positive constant which
may change from line to line and whose dependence on parameters will
always be clear from the context unless otherwise specified.

\subsection{Besov spaces}
\label{ss.Besov.spaces}

\smallskip

For every $s \in [0,1]$, $p \in [1,\infty]$, and $m \in \Z$, we define the volume-normalized Besov seminorm as
\begin{equation}\label{e.Besov.seminorm}
  [f]_{\Besov{p}{\infty}{s}(\cu_m)} \coloneqq
  \sup_{n \leq m} 3^{-s n}
  \biggl( \avsum_{z \in 3^{n-1} \Z^d, z + \cu_n \subseteq \cu_m} \fint_{z + \cu_n} |f(x) - (f)_{z + \cu_n}|^p \, dx \biggr)^{\!\nicefrac{1}{p}}.
\end{equation}
The corresponding volume normalized Besov norm is defined as
\begin{equation}\label{e.Besov.norm}
  \|f\|_{\Besov{p}{\infty}{s}(\cu_m)} \coloneqq
  \bigl(
  3^{-spm}|(f)_{\cu_m}|^p + [f]_{\Besov{p}{\infty}{s}(\cu_m)}^p
  \bigr)^{\nicefrac1p}
  \,.
\end{equation}
There exists~$C(d,s,p)<\infty$ such that, for every~$u \in \Besov{p}{\infty}{s}(\cu_n)$ and~$0<s\leq t\leq 1$, we have the following relations between Besov norms and other classical norms, see~\cite{Triebel1983}:
\begin{equation}
  \label{e.besov.prelim}
  \left\{
  \begin{aligned}
     & C^{-1}  [u]_{\Besov{p}{\infty}{0}(\cu_n)} \leq
    \|u-(u)_{\cu_n}\|_{\underline{L}^p(\cu_n)} \leq C [u]_{\Besov{p}{\infty}{0}(\cu_n)}\,,
    \\ &
    C^{-1} \|u\|_{\underline{L}^p(\cu_n)}         \leq \|u\|_{\Besov{p}{\infty}{0}(\cu_n)}
    \leq C \|u\|_{\underline{L}^p(\cu_n)} \,,
    \\ &
    \|u\|_{\Besov{p}{\infty}{0}(\cu_n)}
    \leq C 3^{sn} \|u\|_{\Besov{p}{\infty}{s}(\cu_n)}
    \leq
    C 3^{tn} \|u\|_{\Besov{p}{\infty}{t}(\cu_n)}
    \\ &
    C^{-1}
    \|u\|_{C^{0,s}(\cu_n)}
    \leq
    \|u\|_{\Besov{\infty}{\infty}{s}(\cu_n)}
    \leq
    C
    \|u\|_{C^{0,s}(\cu_n)}
    \,.
  \end{aligned}
  \right.
\end{equation}
We denote the H\"older conjugate of an exponent~$p\in (1,\infty)$ by~$p' \coloneqq \frac{p}{p-1}$, with obvious changes if $p=1,\infty$. For every~$s \in (0,1]$ and~$p \in [1,\infty]$, we define the dual norm of~$\Besov{p}{\infty}{s}(\cu_n)$ by
\begin{equation}
  \label{e.Besov.dual.norm}
  \|f\|_{\BesovDual{p'}{1}{-s}(\cu_n)}
  \coloneqq
  \sup \biggl \{  \fint_{\cu_n} fg:
  g \in \Besov{p}{\infty}{s}(\cu_n),
  \|g\|_{\Besov{p}{\infty}{s}(\cu_n)} \leq 1 \biggr \}
  \,.
\end{equation}
Observe that we define the dual norm with respect to \emph{all}~$\Besov{p}{\infty}{s}(\cu_n)$ functions, not just those which vanish at the boundary (zero trace functions)---the hat in the notation is meant to distinguish this dual norm from the usual dual norm of~$\Besov{p}{\infty}{-s}(\cu_m)$, which we do not use in this paper. We also introduce a variant of the dual norm by defining the quantity
\begin{equation}\label{e.Besov.dual.sum.norm}
  \|f\|_{\BesovDualSum{p'}{1}{-s}(\cu_n)} \coloneqq
  \sum_{k=-\infty}^n 3^{sk}
  \biggl( \avsum_{z \in 3^{k} \Z^d \cap \cu_n} |(f)_{z + \cu_k}|^{p'} \biggr)^{\!\nicefrac{1}{p'}}
  \,.
\end{equation}
Compared to~$\|f\|_{\BesovDual{p'}{1}{-s}(\cu_n)}$, the quantity~$\|f\|_{\BesovDualSum{p'}{1}{-s}(\cu_n)}$ is easier to compute and upper bounds the former: we have
\begin{equation}\label{e.dual.norms.relation}
  \|f\|_{\BesovDual{p'}{1}{-s}(\cu_n)} \leq 3^{d+s} \|f\|_{\BesovDualSum{p'}{1}{-s}(\cu_n)},
\end{equation}
see~\cite[Lemma A.1]{AK.HC}. The next lemma gives control for the norm in~\eqref{e.Besov.dual.sum.norm} by means of the norm of the dual space of a fractional Sobolev space.
\begin{lemma}
  \label{l.weak.Wsp.bounds.Besov.circ}
  Let \(s,t\in(0,1]\), \(p\in(1,\infty)\) and $q\in [p,\infty]$.
  Assume
  \begin{equation}
    \label{e.ts.condition.corrected}
    t<\min\!\Big\{1-\tfrac{1}{p},\;2s+\tfrac{d}{q}-\tfrac{d}{p}\Big\}.
  \end{equation}
  Set \(\sigma \coloneqq s-\tfrac{t}{2}-\tfrac{d}{2}\big(\tfrac{1}{p}-\tfrac{1}{q}\big)>0\) and \(p'=\tfrac{p}{p-1}\).
  Then there exists a constant \(C=C(d)<\infty\) such that for every
  \(f\in L^1(\cu_0)\cap \Wminusul{-t}{p}(\cu_0)\) we have
  \begin{equation}
    \label{e.weak.Wsp.bounds.Besov.circ}
    \begin{cases}
    \displaystyle
    \ \sum_{k=-\infty}^0 3^{sk}\,
    \biggl( \avsum_{z \in 3^k \Z^d \cap \cu_0} \bigl| (f)_{z+\cu_k} \bigr|^q \biggr)^{\!\frac{1}{2q}}
    \le
    \frac{1}{\cs_{\sigma}}
    \Bigg(\frac{C}{t\,p'\,(1-t p')}\Bigg)^{\!\frac{1}{2p'}}
    \,[f]_{\Wminusul{-t}{p}(\cu_0)}^{\frac12},\quad &\mbox{if} \ \ q\in [p,\infty)\vspace{1.5mm}\\
    \displaystyle
    \ \sum_{k=-\infty}^0 3^{sk}\,
    \biggl( \max_{z\in 3^{k}\mathbb{Z}^{d}\cap \cu_{0}}|(f)_{z+\cu_{k}}|^{\nicefrac{1}{2}} \biggr)
    \le
    \frac{1}{\cs_{\sigma}}
    \Bigg(\frac{C}{t\,p'\,(1-t p')}\Bigg)^{\!\frac{1}{2p'}}
    \,[f]_{\Wminusul{-t}{p}(\cu_0)}^{\frac12},\quad &\mbox{if} \ \ q=\infty,
    \end{cases}
  \end{equation}
  where $\cs_{\sigma} \coloneqq 1 - 3^{-\sigma}$.
\end{lemma}
\begin{proof}
  For each $k\le0$ set $\mathcal{Z}_k\coloneqq 3^k\Z^d\cap\cu_0$ and define the piecewise-constant test function
  \[
    \psi_k \coloneqq \sum_{z\in \mathcal{Z}_k} \alpha_z\,\indc_{z+\cu_k},
    \qquad
    \alpha_z \coloneqq |(f)_{z+\cu_k}|^{p-2}(f)_{z+\cu_k}.
  \]
  A direct calculation gives the identity
  \begin{equation}\label{e.dual.id}
    \int_{\cu_0} f\,\psi_k
    = \avsum_{z\in \mathcal{Z}_k} |(f)_{z+\cu_k}|^p.
  \end{equation}
  By duality between $\Wul{t}{p'}(\cu_0)$ and $\Wminusul{-t}{p}(\cu_0)$,
  \begin{equation}\label{e.dual.bd}
    \int_{\cu_0} f\,\psi_k
    \le \|f\|_{\Wminusul{-t}{p}(\cu_0)}\,\|\psi_k\|_{\Wul{t}{p'}(\cu_0)}.
  \end{equation}

  Since $\psi_k$ is constant on each cube $z+\cu_k$, a standard estimate for such step functions yields, for some dimensional constant $C(d)$,
  \begin{equation}\label{e.step.fn}
    \|\psi_k\|_{\Wul{t}{p'}(\cu_0)}
    \le C(d)\,t^{-\frac{1}{p'}}(1-tp')^{-\frac{1}{p'}}\, 3^{-t k}
    \biggl(\avsum_{z\in \mathcal{Z}_k} |(f)_{z+\cu_k}|^{p}\biggr)^{\!\frac{1}{p'}}.
  \end{equation}
  Combining~\eqref{e.dual.id}--\eqref{e.step.fn} and rearranging gives
  \begin{equation}\label{e.p-bound}
    \biggl(\avsum_{z\in \mathcal{Z}_k} |(f)_{z+\cu_k}|^p\biggr)^{\!\frac{1}{p}}
    \le C(d)\,t^{-\frac{1}{p'}}(1-tp')^{-\frac{1}{p'}}\,3^{-t k}\,\|f\|_{\Wminusul{-t}{p}(\cu_0)}.
  \end{equation}
  Since $p\le q$, the discrete $\ell^p$--$\ell^q$ inequality on the partition $\mathcal{Z}_k$ (of cardinality $|\mathcal{Z}_k|\simeq 3^{-dk}$) gives, for $q<\infty$,
  \[
    \biggl(\avsum_{z\in \mathcal{Z}_k} |(f)_{z+\cu_k}|^q\biggr)^{\!\frac{1}{q}}
    \le |\mathcal{Z}_k|^{\frac{1}{p}-\frac{1}{q}}
    \biggl(\avsum_{z\in \mathcal{Z}_k} |(f)_{z+\cu_k}|^p\biggr)^{\!\frac{1}{p}}
    = 3^{-d k\left(\frac{1}{p}-\frac{1}{q}\right)}
    \biggl(\avsum_{z\in \mathcal{Z}_k} |(f)_{z+\cu_k}|^p\biggr)^{\!\frac{1}{p}},
  \]
  while for $q=\infty$ it is
  $$
  \max_{z\in \mathcal{Z}_{k}}|(f)_{z+\cu_{k}}|\le |\mathcal{Z}_{k}|^{\frac{1}{p}}\biggl(\avsum_{z\in \mathcal{Z}_k} |(f)_{z+\cu_k}|^p\biggr)^{\!\frac{1}{p}}
    = 3^{-\frac{d k}{p}}
    \biggl(\avsum_{z\in \mathcal{Z}_k} |(f)_{z+\cu_k}|^p\biggr)^{\!\frac{1}{p}}.
  $$
  Combining the previous two estimates with~\eqref{e.p-bound} and taking the square root yields
  \begin{equation}\label{e.q-step}
  \begin{cases}
  \displaystyle
   \  \biggl(\avsum_{z\in \mathcal{Z}_k} |(f)_{z+\cu_k}|^q\biggr)^{\!\frac{1}{2q}}
    \le C_1\,3^{-\frac{t}{2}k}\,3^{-\frac{d}{2}k\left(\frac{1}{p}-\frac{1}{q}\right)}\|f\|_{\Wminusul{-t}{p}(\cu_0)}^{\frac12}\vspace{1.5mm}\\
    \displaystyle
    \ \max_{z\in \mathcal{Z}_{k}}|(f)_{z+\cu_{k}}|^{\nicefrac{1}{2}}\le C_1\,3^{-\frac{tk}{2}}\,3^{-\frac{dk}{2p}}\|f\|_{\Wminusul{-t}{p}(\cu_0)}^{\frac12}
    \end{cases}
  \end{equation}
  where $C_1=C_1(d,t,p)$ equals the square root of the constant $C(d)t^{-\nicefrac{1}{p'}}(1-tp')^{-\nicefrac{1}{p'}}$ in~\eqref{e.p-bound}.

  Finally multiply~\eqref{e.q-step} by $3^{s k}$ and sum over $k\le0$. The geometric series
  \[
    \sum_{k=-\infty}^0 3^{\bigl(s-\frac{t}{2}-\frac{d}{2}(\frac{1}{p}-\frac{1}{q})\bigr)k}
    = \frac{1}{\cs_\sigma}
  \]
  converges precisely when $\sigma\coloneqq s-\frac{t}{2}-\frac{d}{2}\big(\frac{1}{p}-\frac{1}{q}\big)>0$ - if $q=\infty$, take $\nicefrac{1}{q}=0$. Collecting constants yields the stated estimate.
\end{proof}

We next present a proof of the Sobolev embedding~$\Wul{1}{q} \hookrightarrow \Besov{p}{\infty}{s}$.

\begin{lemma}\label{l.besov.by.gradient}
  Let $1 < p < \infty$, $n \in \N$ and let $q \in [p_\ast, p]$ with~$p_\ast = \frac{dp}{d+p}$.
  Then, for every $s \in (0,1-\nicefrac{d}{q}+\nicefrac{d}{p}]$, there exists $C(d) < \infty$ such that, for every $g \in \Wul{1}{q}(\cu_n)$,
  \begin{equation*}
    [g]_{\Besov{p}{\infty}{s}(\cu_n)} \leq C 3^{(1-s)n} \|\nabla g\|_{\underline{L}^{q}(\cu_n)}.
  \end{equation*}
\end{lemma}

\begin{proof}
  Recall the definition of the Besov seminorm in~\cref{e.Besov.seminorm}:
  \begin{equation*}
    [g]_{\Besov{p}{\infty}{s}(\cu_n)}
    =
    \sup_{k \leq n} 3^{-sk}
    \biggl(
    \avsum_{z \in 3^k \Z^d \cap \cu_n} \|g - (g)_{z+\cu_k}\|_{\underline{L}^{p}(z+\cu_k)}^{p}
    \biggr)^{\!\nicefrac{1}{p}}
    \,.
  \end{equation*}
  The standard Sobolev-Poincaré inequality applied to each term in the sum shows that there exists~$C(d)<\infty$ such that, for every~$q \in [p_\ast, p]$,
  \begin{equation*}
    \sup_{k \leq n} 3^{-sk}
    \biggl(
    \avsum_{z \in 3^k \Z^d \cap \cu_n} \|g - (g)_{z+\cu_k}\|_{\underline{L}^{p}(z+\cu_k)}^{p}
    \biggr)^{\!\nicefrac{1}{p}}
    \leq
    C
    \sup_{k \leq n} 3^{-sk+k}
    \biggl(
    \avsum_{z \in 3^k \Z^d \cap \cu_n} \|\nabla g\|_{\underline{L}^{q}(z+\cu_k)}^{p}
    \biggr)^{\!\nicefrac{1}{p}}\,.
  \end{equation*}
  Using the discrete $\ell^t$--$\ell^1$ inequality with $t = \nicefrac{p}{q}$ on a partition of size~$N = 3^{d(n-k)}$, we have, for every~$k \leq n$,
  \begin{equation*}
    \biggl(
    \avsum_{z \in 3^k \Z^d \cap \cu_n} \|\nabla g\|_{\underline{L}^{q}(z+\cu_k)}^{p} \biggr)^{\!\nicefrac{q}{p}}
    \leq
    3^{d(n-k)(1-\nicefrac{q}{p})}
    \avsum_{z \in 3^k \Z^d \cap \cu_n} \|\nabla g\|_{\underline{L}^{q}(z+\cu_k)}^{q}
      = 3^{d(n-k)(1-\nicefrac{q}{p})} \|\nabla g\|_{\underline{L}^{q}(\cu_n)}^{q}
    \,.
  \end{equation*}
  We therefore obtain
  \begin{equation*}
    [g]_{\Besov{p}{\infty}{s}(\cu_n)} \leq C \sup_{k \leq n} 3^{(1-s)k+d(n-k)(\nicefrac{1}{q}-\nicefrac{1}{p})} \|\nabla g\|_{\underline{L}^{q}(\cu_n)}
    =
    C 3^{dn(\nicefrac{1}{q}-\nicefrac{1}{p}) } \sup_{k \leq n} 3^{k (1-s - d ( \nicefrac{1}{q} -\nicefrac{1}{p}))}
    \|\nabla g\|_{\underline{L}^{q}(\cu_n)}
    \, .
  \end{equation*}
  By assumption, we have that the exponent of~$3^k$ in the supremum is nonnegative. We deduce that the supremum is attained at $k=n$, and we conclude that~$[g]_{\Besov{p}{\infty}{s}(\cu_n)} \leq C 3^{(1-s)n}\|\nabla g\|_{\underline{L}^{q}(\cu_n)}$, which completes the proof.
\end{proof}

The following lemma is our version of the Sobolev-Poincaré inequality in Besov spaces, it also contains a product rule estimate, which is essentially~\cite[Lemma A.3]{AK.HC}. 
\begin{lemma}
  \label{l.besov.poincare}
  There exists a constant~$C(d) < \infty$ such that, for every~$n \in \Z$,~$s\in (0,1)$ and~$u \in H^1_\a(\cu_n)$, we have
  \begin{equation} \label{e.besov.poincare}
    \|u-(u)_{\cu_n}\|_{\underline{L}^{2^\ast_s}(\cu_n)}
    \leq
    C 3^{(1-s)n} \|\nabla u\|_{\BesovDualSum{2}{1}{-s}(\cu_n)}
    \quad \mbox{with} \quad 2^\ast_s = \frac{2d}{d-2(1-s)}\,.
  \end{equation}
  Moreover, for every $\varphi \in \Wul{2}{\infty}(\cu_n)$,
  \begin{equation} \label{e.besov.product.rule}
    \|(u-(u)_{\cu_n}) \nabla \varphi\|_{\Besov{2}{\infty}{s}(\cu_n)}
    \leq
    C 3^n \|\nabla \varphi\|_{\Wul{1}{\infty}(\cu_n)} \|\nabla u\|_{\BesovDualSum{2}{1}{s-1}(\cu_n)}.
  \end{equation}
\end{lemma}
\begin{proof}
  Without loss of generality, we can assume that~$(u)_{\cu_n} = 0$.

  \emph{Step 1: The Sobolev embedding.}
  Let $q \coloneqq 2^\ast_s$ and let $w \in W^{1,q'}(\cu_n)$ solve
  \begin{equation}
    -\Delta w = |u|^{q-2}u - (|u|^{q-2}u)_{\cu_n} \quad\text{in }\cu_n,
    \qquad
    \partial_\nu w=0 \text{ on }\partial\cu_n.
    \label{e.w.PDE}
  \end{equation}
  Testing~\eqref{e.w.PDE} with $u$ and using the duality of our Besov spaces yields
  \begin{equation}
    \|u\|_{\underline{L}^q(\cu_n)}^q
    = \fint_{\cu_n} \nabla u \cdot \nabla w
    \leq \|\nabla u\|_{\BesovDual{2}{1}{-s}(\cu_n)} \;\|\nabla w\|_{\Besov{2}{\infty}{s}(\cu_n)}.
    \label{e.u.norm.duality}
  \end{equation}
  By Lemma~\ref{l.besov.by.gradient}, for $s \leq 1 - \nicefrac{d}{q'} + \nicefrac{d}{2}$ we have
  \begin{equation}
    [\nabla w]_{\Besov{2}{\infty}{s}(\cu_n)} \leq C 3^{(1-s)n} \|\nabla w\|_{\Wul{1}{q'}(\cu_n)}.
    \label{e.besov.gradient.w}
  \end{equation}
  Moreover, the mean term is controlled by Sobolev-Poincaré
  \begin{equation}
    3^{-sn} |(\nabla w)_{\cu_n}| \leq 3^{-sn}\|\nabla w\|_{\underline{L}^{q'}(\cu_n)} \leq C 3^{(1-s)n} \|\nabla w\|_{\Wul{1}{q'}(\cu_n)}.
    \label{e.mean.gradient.w}
  \end{equation}
  Classical Calderón–Zygmund estimates give
  \begin{equation}
    \|\nabla w\|_{\Wul{1}{q'}(\cu_n)} \leq C \|u\|_{\underline{L}^q(\cu_n)}^{q-1}.
    \label{e.CZ.estimate}
  \end{equation}
  Combining~\eqref{e.u.norm.duality}--\eqref{e.CZ.estimate} yields
  \begin{equation}
    \|u\|_{\underline{L}^q(\cu_n)}^q
    \leq C 3^{(1-s)n} \|\nabla u\|_{\BesovDual{2}{1}{-s}(\cu_n)} \|u\|_{\underline{L}^q(\cu_n)}^{q-1},
    \label{e.final.sobolev}
  \end{equation}
  and dividing both sides of~\eqref{e.final.sobolev} by $\|u\|_{\underline{L}^q(\cu_n)}^{q-1}$ (when nonzero) yields~\eqref{e.besov.poincare}.

  \emph{Step 2: Embedding of~$\BesovDualSum{2}{1}{s-1}$ into~$\Besov{2}{\infty}{s}$.}
  We aim to prove the inequality
  \begin{equation}\label{e.besov.embedding}
    \|u-(u)_{\cu_n}\|_{\Besov{2}{\infty}{s}(\cu_n)} \leq C \|\nabla u\|_{\BesovDualSum{2}{1}{s-1}(\cu_n)}.
  \end{equation}
  Using Step~1 with $s$ replaced by $1$, we have for any box $z+\cu_k \subseteq \cu_n$,
  \begin{equation*}
    \|u-(u)_{z+\cu_k}\|_{\underline{L}^{2}(z+\cu_k)} \leq C \sum_{m=-\infty}^k 3^m \biggl( \avsum_{y \in 3^m \Z^d \cap \cu_n} |(\nabla u)_{y+\cu_m}|^{2} \biggr)^{\!\nicefrac{1}{2}}.
  \end{equation*}
  Therefore, applying the definition of the Besov seminorm, the above and Hölder's inequality yield
  \begin{align*}
    \|u-(u)_{\cu_n}\|_{\Besov{2}{\infty}{s}(\cu_n)}
    & =
    \sup_{k \leq n} 3^{-sk}
    \biggl(
    \avsum_{z \in 3^k \Z^d \cap \cu_n} \|u - (u)_{z+\cu_k}\|_{\underline{L}^{2}(z+\cu_k)}^{2}
    \biggr)^{\!\nicefrac{1}{2}}
    \\
    & \leq
    \sup_{k \leq n} 3^{-sk}
    \biggl(
    \avsum_{z \in 3^k \Z^d \cap \cu_n} \biggl (\sum_{m=-\infty}^k 3^m \biggl (\avsum_{y \in 3^m \Z^d \cap z+\cu_k} |(\nabla u)_{y+\cu_m}|^{2} \biggr)^{\!\nicefrac{1}{2}} \biggr)^{2}
    \biggr)^{\!\nicefrac{1}{2}}
    \\
    & \leq
    \sum_{m=-\infty}^n 3^{(1-s)m} \biggl (\avsum_{y \in 3^{m} \Z^d \cap \cu_n} |(\nabla u)_{y+\cu_{m}}|^{2} \biggr)^{\!\nicefrac{1}{2}} 
    =
    \|\nabla u\|_{\BesovDualSum{2}{1}{s-1}(\cu_n)},
  \end{align*}
  which is~\eqref{e.besov.embedding}.

  \emph{Step 3.} We now prove~\eqref{e.besov.product.rule}. Observe that
  by the definition of the Besov seminorm, we have
  \begin{equation*}
    [u \nabla \varphi]_{\Besov{2}{\infty}{s}(\cu_n)}
    =
    \sup_{k \leq n} 3^{-sk}
    \biggl(
    \avsum_{z \in 3^k \Z^d \cap \cu_n} \|u \nabla \varphi - (u \nabla \varphi)_{z+\cu_k}\|_{\underline{L}^{2}(z+\cu_k)}^{2}
    \biggr)^{\!\nicefrac{1}{2}}.
  \end{equation*}
  Using the triangle inequality inside the average, we get
  \begin{multline*}
    \|u \nabla \varphi - (u \nabla \varphi)_{z+\cu_k}\|_{\underline{L}^{2}(z+\cu_k)} \\
    \leq
    C \|\nabla \varphi\|_{\underline{L}^{\infty}(z+\cu_k)} \|u-(u)_{z+\cu_k}\|_{\underline{L}^{2}(z+\cu_k)} 
    +
    C \|u\|_{\underline{L}^{2}(z+\cu_k)} \|\nabla \varphi - (\nabla \varphi)_{z+\cu_k}\|_{\underline{L}^{2}(z+\cu_k)}.
  \end{multline*}
  From the above we deduce that
  \begin{equation}
    [u \nabla \varphi]_{\Besov{2}{\infty}{s}(\cu_n)}
    \leq
    C \|\nabla \varphi\|_{\underline{L}^{\infty}(\cu_n)} [u]_{\Besov{2}{\infty}{s}(\cu_n)}
    +
    C \|\nabla \varphi\|_{\Wul{1}{\infty}(\cu_n)} \|u\|_{\underline{L}^{2}(\cu_n)}.
  \end{equation}
  We can now use~\eqref{e.besov.embedding} to bound the right-hand side and obtain the desired result.
\end{proof}

\subsection{Weighted Sobolev spaces}
\label{ss.sobolev.spaces}

In this section we assume that the domain~$U \subseteq\Rd$ is a \emph{bounded} Lipschitz domain and we assume that the coefficient field~$\a$ is a positive semi-definite symmetric matrix such that $\a(x)$ is invertible for a.e.~$x$, satisfying~$\a,\a^{-1} \in L^1(U;\R^{d\times d})$. We define the norm
\begin{equation}
  \label{e.H1s}
  \| u \|_{H^1_\a(U)} \coloneqq \Bigl( \| u \|_{L^2(U)}^2 + \int_U \nabla u \cdot \a \nabla u\Bigr)^{\nicefrac12} \,.
\end{equation}
By H\"older's inequality, we have, for every~$u \in C^\infty (U)$, that~$\| u \|_{H^1_\a(U)} < \infty$. We define the function space~$H^1_\a(U)$ as the completion of~$C^\infty(U)$ with respect to the above norm. Again by H\"older's inequality, we have that
\begin{equation}
  \label{e.grad.and.flux.in.Lone}
  u\in H^1_\a(U)
  \implies
  \nabla u , \,\a\nabla u \in L^1(U)
  \,.
\end{equation}
Indeed,~$u\in H^1_\a(U)$ implies that~$\a^{\nicefrac12} \nabla u \in L^2(U)$ and the assumption of~$\a,\a^{-1}\in L^1(U)$ implies that~$\a^{\nicefrac12},\a^{-\nicefrac12} \in L^2(U)$. Thus, Cauchy-Schwarz gives the implication~\eqref{e.grad.and.flux.in.Lone}.
Following the argument in~\cite[Theorem 1.11]{KO84}, the space~$H^1_\a(U)$ is a complete Hilbert space for every Lipschitz domain~$U\subseteq\Rd$ and~$\a$ satisfying~$\a,\a^{-1} \in L^1(U)$. Notice that the paper~\cite{KO84} considers only the case when~$\a$ is scalar, but the proof generalizes to the general (matrix-valued) case.  We also define the subspace of~``trace zero'' functions by
\begin{equation*}
  H^1_{\a,0}(U) \coloneqq  \mbox{closure of~$C^\infty_c (U)$ with respect to~$\| \cdot\|_{H^1_{\a}(U)}$}\,.
\end{equation*}
We denote the dual space of~$H^1_{\a,0}(U) $ by~$H_\a^{-1}(U)$.
We define the dual norm on $H^{-1}_{\a}(U)$ by
\begin{equation*}
  \| f \|_{H^{-1}_{\a}(U)}  \coloneqq  \sup \bigl\{ \langle u,f \rangle \,:\, u\in H^1_{\a,0}(U), \ \| u \|_{H^1_{\a}(U)} \leq 1 \bigr\},
\end{equation*}
and, when convenient, we abusively write $\int_U u f$ in place of the pairing $\langle u,f\rangle$ for $u\in H^1_{\a,0}(U)$ and $f\in H^{-1}_\a(U)$.

The linear subspace of~$H^1_{\a}(U)$ consisting of solutions of the equation~$\nabla\cdot \a\nabla u =0$ is denoted by
\begin{equation*}
  \A(U;\a)  \coloneqq
  \bigl\{
  u \in H^1_\a(U) \,:\,
  \nabla \cdot \a\nabla u = 0 \ \mbox{in} \ U
  \bigr\}
  \,.
\end{equation*}
Here the equation~$\nabla \cdot \a\nabla u=0$ is to be understood in the sense of distributions; that is,
\begin{equation*}
  \int_U \nabla \psi \cdot \a\nabla u
  = 0 \,,
  \qquad \forall \psi \in C^\infty_c(U)\,.
\end{equation*}
We write~$\A(U)$ instead of~$\A(U;\a)$ when the coefficient field~$\a$ is clear from the context.  We denote the space of weak subsolutions as $\A_-(U) \coloneqq \{ u \in H^1_\a(U) : -\nabla \cdot ( \a \nabla u ) \leq 0 \text{ in } U \}$, where the meaning of the differential inequality is the following:
\begin{equation*}
  u \in \A_-(U)  \,, \quad  v \in C^\infty_c(U)\,, \quad v \geq 0
  \quad
  \implies
  \quad
  \int_U \nabla v \cdot \a\nabla u \leq 0
  \,.
\end{equation*}

\smallskip

Next, consider the Dirichlet problem
\begin{equation}
  \label{e.H1s.Dirichlet}
  \left\{
  \begin{aligned}
    & -\nabla \cdot \a \nabla v = f
    & \mbox{in} & \ U \,,
    \\
    & v = 0 & \mbox{on} & \ \partial U\,,
  \end{aligned}
  \right.
\end{equation}
with $f\in H^{-1}_{\a}(U)$. Define the bilinear form $B(u,v) \coloneqq \int_U \nabla u \cdot \a \nabla v$.
By Cauchy-Schwarz and the definition of the $H^1_\a$-norm, $B$ is bounded on $H^1_{\a,0}(U)\times H^1_{\a,0}(U)$ and $B$ is coercive on $H^1_{\a,0}(U)$ in the energy sense. Finally, for $f\in H^{-1}_{\a}(U)$ the linear functional $v\mapsto\langle v,f\rangle$ is continuous on $H^1_{\a,0}(U)$ by the definition of the dual. Hence, by the Lax-Milgram theorem there exists a unique solution $u\in H^1_{\a,0}(U)$ of the weak formulation
\begin{equation*}
  B(u,v)=\langle v,f\rangle,\qquad\forall v\in H^1_{\a,0}(U),
\end{equation*}
i.e.~a unique weak solution of the boundary-value problem~\eqref{e.H1s.Dirichlet}.

By the solvability above one also obtains the usual characterization of the dual: every element of $H^{-1}_{\a}(U)$ can be represented (in the distributional sense) as a divergence of an $L^2$--vector field weighted by $\a^{\nicefrac12}$, namely
\begin{equation*}
  H^{-1}_{\a}(U)
  =
  \bigl\{
  \nabla \cdot (\a^{\nicefrac12} \f) \,:\, \f \in L^2(U;\Rd)
  \bigr\}.
\end{equation*}

\smallskip

In Appendix~\ref{app.sobolev}, we prove some necessary testing identities for solutions and subsolutions.

\subsection{Cones of subsolutions}

Given a Lipschitz domain~$U$ in~$\Rd$, define the following subsets of~$H^1_\a(U)$:
\begin{equation}
  \label{e.admissible.cones}
  \left\{
  \begin{aligned}
     & \baconsol\!(U)  \coloneqq  \A(U)\,,\\
     & \baconsub\!(U) \coloneqq  \A_-(U)\,, \\
     & \baconconv\!(U) \coloneqq  \bigl \{ \Phi(u) \in H^1_\a(U) \, : \, u \in \A(U) \,, \; \Phi: u(U) \to \R \text{ is~$C^1$, convex and monotone} \bigr\}\,, \\
     & \bacontrunc\!(U) \coloneqq  \bigl\{ (u-k)_+ \, : \, u \in \A(U)\,, \; k \in \R\bigr\}\,,\\
     & \baconpow\!(U) \coloneqq  \bigl\{ \Phi(u) \in H^1_\a(U) \, : \, u \in \A(U)\,, u \geq 0\,, \; \Phi'(u) = C u^p\,,\; C \in \R\,, \; p \in \R\,, \;\\ 
     & \qquad \qquad  \qquad \qquad  \qquad \qquad \quad
    \Phi:u(U) \to \R \text{ is convex and monotone} \bigr\}\,,
    \\
    & \baconmax\!(U) \coloneqq  H^1_\a(U)\,.
  \end{aligned}
  \right.
\end{equation}
We call the six families of cones in~\eqref{e.admissible.cones}, and unions of them, the \emph{admissible cones}, and let~$\mathcal{C}(U)$ denote the set of admissible cones.
It is easy to see that each one of the above satisfies the following properties:
\begin{enumerate}[label=(\roman*)]
  \item \textbf{Cone property:} For every~$u \in \bacon(U)$ and~$t \in \R_+$, we have $tu \in \bacon(U)$.
  \item \textbf{Restriction stability:} If $u \in \bacon(U)$ and $V \subseteq U$ is a Lipschitz domain, then $u|_V \in \bacon(V)$.
\end{enumerate}

\smallskip

The cone $\baconsub(U)$ is special as it is an admissible family of convex cones, furthermore each cone is stable under convex modifications in the following sense: Let $u \in \baconsub(U)$ and take a non-decreasing convex function $\Phi$ such that $\Phi(u),\Phi'(u) \in H^1_\a(U)$, test the equation with $\eta = (\Phi(u+h)-\Phi(u)) \varphi$ for~$h>0$ and~$\varphi \in C^\infty_c(U)$ and $\varphi \geq 0$, then
\begin{align*}
  0 &
  \geq \lim_{h\downarrow 0}  \frac1h \fint_U \a \nabla u \cdot \nabla \bigl( (\Phi(u+h) - \Phi(u)) \varphi\bigr)
  \notag \\ &
  = \lim_{h\downarrow 0}  \frac1h \fint_U \a \nabla u \cdot \bigl( (\Phi'(u+h) - \Phi'(u)) \nabla u \varphi + (\Phi(u+h) - \Phi(u)) \nabla \varphi \bigr)
  \notag \\ &
  \geq \lim_{h\downarrow 0}  \frac1h
  \fint_U \a \nabla u \cdot (\Phi(u+h) - \Phi(u)) \nabla \varphi
  =
  \fint_U \a \nabla \Phi(u) \cdot \nabla \varphi
  \,.
\end{align*}
The above testing can be made rigorous using Lemma~\ref{l.product.rule}.
Therefore,~$\Phi(u) \in \baconsub(U)$. From this calculation we obtain that~$\baconconv(U) \subseteq \baconsub(U) \cup (-\baconsub\!(U))$. Furthermore, it is clear that both $\bacontrunc(U)$ and $\baconpow(U)$ are subsets of $\baconconv(U)$. Also, each of the above cones belongs to the maximal cone~$ \baconmax(U)= H^1_\a(U)$. 

\subsection{Coarse-grained matrices}

Consider a Lipschitz domain~$U\subseteq \Rd$ and a cone~$\bacon(U) \in \mathcal{C}(U)$.
For each~$p,q\in\Rd$, we define
\begin{equation*}
  \a(U,p\, ;\bacon)  \coloneqq  \sup_{u \in \bacon(U)} \fint_{U} \Bigl( - \nabla u \cdot \a \nabla u + 2 p \cdot \a \nabla u \Bigr)
\end{equation*}
and
\begin{equation*}
  \b(U,q\, ;\bacon) \coloneqq \sup_{u \in \bacon(U)} \fint_{U} \Bigl( -  \nabla u \cdot \a \nabla u + 2 q \cdot \nabla u \Bigr)\,.
\end{equation*}
Clearly~$ \a(U,p\, ;\bacon) $ and~$\b(U,q\,;\bacon)$ are increasing with respect to the cones.

\smallskip

In the case~$\bacon =   \baconsol$, coarse-graining operators are the central objects in the theory developed in~\cite{AKM.Book,AK.HC}. Then there exist symmetric positive matrices~$\a(U)$ and~$\a_*(U)$ such that, for every~$p,q\in \Rd$,
\begin{equation*}
  \a(U,p\, ;\baconsol)
  = p \cdot \a(U) p
  \quad \mbox{and} \quad
  \b(U,q\,;\baconsol)
  = q \cdot \a_*^{-1}(U) q
  \,,
\end{equation*}
that is, they are quadratic functions of~$p$ and~$q$, respectively. Moreover, we have the following ordering for the matrices:
\begin{equation}
  \label{e.a.vs.astar}
  \a_*(U)
  \leq
  \a(U)
  \,.
\end{equation}
Since~$\baconsol(U) \subseteq H_\a^1(U)$ and, by the first variation, the maximizer of~$\b(U,q\, ; H_\a^1(U))$ belongs to~$\baconsol(U)$, we in fact have
\begin{equation*}
  \b(U,q\,;\baconsol)
  \leq
  \b(U,q\,;\bacon)
  \leq
  \b(U,q\,;H_\a^1(U)) =
  \b(U,q\,;\baconsol)
\end{equation*}
for each~$\bacon(U)\in \mathcal{C}(U)$.
Thus, the inequalities above are actually equalities and, hence, for every~$q \in \Rd$ and for each~$\bacon(U)\in \mathcal{C}(U)$,
\begin{equation}
  \label{e.bacon.astar}
  \b(U,q\, ;\bacon)  = q \cdot \a_*^{-1}(U) q
  \,.
\end{equation}
The distinction between the different cones~$\bacon$ is relevant only for~$\a(U,p\,;\bacon)$, and therefore we will drop~$\b(U,q\, ;\bacon)$ from the notation in favor of~$\a_*^{-1}(U)$.
We define the norm of~$p \mapsto \a(U,p\, ;\bacon)$ by
\begin{equation}
  \| \a(U;\bacon)  \| \coloneqq \sup_{|p|=1} \a(U,p\, ;\bacon)
  \,.
\end{equation}

\smallskip

The coarse-grained operators allow us to control the average fluxes and gradients of functions in the cones $\bacon(U)$ in terms of their energy. This is one of the cornerstones of coarse-graining theory developed in~\cite{AKM.Book,AK.HC}.

\begin{lemma}\label{l.cone.homogeneity}
  Let~$\bacon(U)\in \mathcal{C}(U)$.
  We have the following Rayleigh quotient representations: for every~$p,q\in\Rd$,
  \begin{equation}
    \label{e.Rayleigh}
    \a(U,p\, ;\bacon) = \sup_{u \in \bacon(U)} \frac{(\fint_U \a \nabla u \cdot p)_+^2}{ \fint_U \nabla u \cdot \a \nabla u}
    \quad \mbox{and} \quad
    q\cdot \a_*^{-1}(U) q = \sup_{u \in \bacon(U)} \frac{(\fint_U \nabla u \cdot q)^2}{ \fint_U \nabla u \cdot \a \nabla u}
    \,.
  \end{equation}
  In particular, the map~$p \mapsto \a(U, p\, ;\bacon(U))$ is homogeneous of degree~$2$, that is,
  \begin{equation}
    \label{e.a.homogeneity}
    \a(U, t p\, ;\bacon) = t^2 \a(U,p\, ;\bacon)
    \,.
  \end{equation}
  The coarse-grained quantities satisfy, for every~$p\in\Rd$, the bounds
  \begin{equation}
    \label{e.a.norm.bounds}
    p \cdot \bigl( \a^{-1} \bigr)_U^{-1} p
    \leq
    p\cdot \a_*(U) p
    \leq
    \a(U,p;\bacon)
    \leq
    p \cdot \bigl( \a\bigr)_U p
    \,.
  \end{equation}
  Finally, for every~$u\in \bacon(U)$, we have
  \begin{equation}
    |(\a \nabla u)_{U}|
    \leq
    \| \a(U;\bacon)  \|^{\nicefrac 12}  \| \a^{\nicefrac12} \nabla u \|_{\underline{L}^2(U)}
    \qquad \mbox{and} \qquad
    |(\nabla u)_{U}|  \leq | \a_*^{-1}(U)  |^{\nicefrac 12}  \| \a^{\nicefrac12} \nabla u \|_{\underline{L}^2(U)}
    \,.
    \label{e.coarse.graining.ineq}
  \end{equation}
\end{lemma}
\begin{proof}
  For any fixed~$t \in \R_+$ and~$u \in \bacon(U)$ we have, since $\bacon(U)$ is a cone, that~$t u \in \bacon(U)$. Thus,
  \begin{equation*}
    \fint_{U} \Bigl( - t^2 \nabla u \cdot \a \nabla u + 2 t p \cdot \a \nabla u \Bigr)
  \end{equation*}
  is maximized at
  \begin{equation*}
    t = \frac{\bigl (\fint_U \a \nabla u \cdot p\bigr)_+}{\fint_U \nabla u \cdot \a \nabla u}
    \,.
  \end{equation*}
  Plugging this back in we get
  \begin{equation*}
    \sup_{t \in \R_+} \fint_{U} \Bigl( -t^2 \nabla u \cdot \a \nabla u + 2 t p \cdot \a \nabla u \Bigr)
    =
    \frac{(\fint_U \a \nabla u \cdot p)_+^2}{\fint_U \nabla u \cdot \a \nabla u}
    \,.
  \end{equation*}
  That is, we can rewrite
  \begin{equation*}
    \a(U,p\, ;\bacon)  = \sup_{u \in \bacon(U)} \frac{(\fint_U \a \nabla u \cdot p)_+^2}{\fint_U \nabla u \cdot \a \nabla u} \,.
  \end{equation*}
  From the above it is immediate that $ \a(U,tp\, ;\bacon) = t^2  \a(U,p\, ;\bacon)$ for every $t \in \R_+$. The last and the first norm bounds in~\eqref{e.a.norm.bounds} follow by H\"older's inequality and~\eqref{e.Rayleigh}.  Finally,~\eqref{e.coarse.graining.ineq} is an immediate consequence of~\eqref{e.Rayleigh} after taking supremum over~$|p|=1$ and~$|q|=1$.
\end{proof}
Finally, we note the following subadditivity property of the coarse-grained operators: Let $U_1,\ldots,U_n$ be a partition of $U$, then
\begin{equation} \label{e.subadditivity}
  \| \a(U;\bacon)  \| \leq \sum_{i=1}^n \frac{|U_i|}{|U|} \| \a(U_i;\bacon)  \|
  \quad \mbox{and} \quad
  \a_*^{-1}(U)  \leq \sum_{i=1}^n \frac{|U_i|}{|U|}\a_*^{-1}(U_i)\,.
\end{equation}
This is an immediate consequence of the definition of the assumption that the cone family is stable under restrictions to subsets.

\subsection{Coarse-grained ellipticity}
\label{ss.coarse.grained.ellipticity}

For each~$s \in (0,\infty)$ we set~$\cs_s \coloneqq 1-3^{-s}$.

\begin{definition}[Coarse-grained ellipticity]
  \label{def.cg.ellipticity}
  Let $\bacon(U)\in \mathcal{C}(U)$. For every $s \in (0,\infty)$, $q \in [1,\infty)$, $m \in \Z$, we define the multiscale Besov-type ellipticity constants as
  \begin{equation}\label{eq:besov-norms}
    \left\{
    \begin{aligned}
       & \Lambda_{s}(\cu_m \,; \bacon)
      \coloneqq
      \biggl( \cs_s \sum_{k=-\infty}^m 3^{-s(m-k)} \max_{z \in 3^k \Z^d \cap \cu_m} \| \a(z+\cu_k;\bacon)  \|^{\nicefrac{1}{2}} \biggr)^{\! 2}\,,
      \\ &
      \lambda_{s}(\cu_m)
      \coloneqq
      \biggl( \cs_s \sum_{k=-\infty}^m 3^{-s(m-k)} \max_{z \in 3^k \Z^d \cap \cu_m} | \a_*^{-1}(z+\cu_k) |^{\nicefrac{1}{2}} \biggr)^{\! -2}\,.
    \end{aligned}
    \right.
  \end{equation}
  We define, for each~$s,t\in (0,1)$ with~$s+t<1$, the \emph{coarse-grained ellipticity ratio} to be
  \begin{equation}
    \Theta_{s,t}(\cu_m; \bacon)
    \coloneqq \frac{\Lambda_s(\cu_m; \bacon)}{\lambda_t(\cu_m)} \,.
    \label{e.cg.ellipticity.ratio}
  \end{equation}
  We say that the coefficient field~$\a(\cdot)$ is \emph{coarse-grained elliptic in~$\cu_m$} if there exist~$s,t\in(0,1)$ such that~$s+t<1$ and~$\Theta_{s,t}(\cu_m; \baconpow \cup \bacontrunc) < \infty$.
\end{definition}

The ellipticity constants $\Lambda_s(\cu_m;\bacon)$ and $\lambda_s(\cu_m)$ are monotone in $s$ but in opposite directions. We have, for every $0 < s\leq t \leq 1$, and cones~$\tilde{\bacon} \subseteq \bacon$ as in \cref{e.admissible.cones},
\begin{equation}\label{e.ellipticity.monotone}
  \lambda_s(\cu_m)
  \leq
  \lambda_t(\cu_m)
  \leq
  \Lambda_t(\cu_m;\tilde{\bacon})
  \leq
  \Lambda_s(\cu_m;\bacon)
  \,.
\end{equation}
This is immediate from~\eqref{e.a.norm.bounds} and the definitions.
In particular,~$\Theta_{s,t}(\cu_m; \bacon)  \geq 1$ for all~$s,t$. We also note that the ellipticity constants satisfy the scaling bounds, for any $k \leq m$,
\begin{equation}
  \label{e.ellipticity.scales}
  \left\{
  \begin{aligned}
     &
    \max_{z \in 3^k \Z^d \cap \cu_m} \Lambda_s(z+\cu_k; \bacon)
    \leq
    3^{2s(m-k)} \Lambda_s(\cu_m; \bacon)\,,
    \\ &
    \max_{z \in 3^k \Z^d \cap \cu_m} \lambda_s^{-1}(z+\cu_k)
    \leq
    3^{2s(m-k)} \lambda_s^{-1}(\cu_m)\,.
  \end{aligned}
  \right.
\end{equation}
Similarly, for every~$k \in \Z$ with~$k \leq m$ and for cones~$\tilde{\bacon} \subseteq \bacon$,
\begin{equation}
  \label{e.Theta.scaling}
  \max_{z \in 3^k \Z^d \cap \cu_m}\Theta_{s,t}(z+\cu_k; \tilde{\bacon}) \leq 3^{2(s+t)(m-k)} \Theta_{s,t}(\cu_m; \bacon).
\end{equation}
We will often drop the dependence on the cone~$\bacon$ in the notation, when it is clear from the context.

We end this section with a lemma giving sufficient conditions for coarse-grained ellipticity in terms of negative Sobolev regularity of the coefficient field and its inverse.

\begin{lemma}
  Let $p,q \in (1,\infty)$, $\alpha,\beta \in [0,1)$ satisfy $\alpha < 1 - 1/p$, $\beta < 1 - 1/q$ and assume that
  \begin{equation}
    \label{e.sigma.condition}
    \tilde \sigma \coloneqq 1 - \frac{d}{2}\left(\frac{1}{p}+\frac{1}{q}\right) - \frac{\alpha+\beta}{2} > 0.
  \end{equation}
  Fix $\varepsilon \in (0,1-\tilde \sigma)$. Then there exist~$s,t \in (0,1]$ such that $s+t = 1 - \varepsilon/2$, and
  \begin{equation}
    \sigma_1 \coloneqq s-\tfrac{\alpha}{2}-\tfrac{d}{2p} > 0,
    \quad
    \sigma_2 \coloneqq t-\tfrac{\beta}{2}-\tfrac{d}{2q} > 0.
  \end{equation}
  If $\a \in L^1(\cu_0) \cap W^{-\alpha,p}(\cu_0)$, and $\a^{-1} \in L^1(\cu_0) \cap W^{-\beta,q}(\cu_0)$, then $\a$ is coarse-grained elliptic in~$\cu_0$ with parameters~$(s,t)$ (in particular $\Theta_{s,t}(\cu_0; \baconpow \cup \bacontrunc) \leq \Theta_{s,t}(\cu_0;\baconmax) < \infty$) and
  \begin{equation}
    \label{e.ellipticity.ratio.bound}
    \Theta_{s,t}(\cu_0; \baconmax)
    \leq \frac{\cs_s^2 \cs_t^2}{\cs_{\sigma_1}^2\cs_{\sigma_2}^2}
    \Bigg(\frac{C}{\alpha\,p'\,(1-\alpha p')}\Bigg)^{\!\frac{1}{p'}}
    \Bigg(\frac{C}{\beta\,q'\,(1-\beta q')}\Bigg)^{\!\frac{1}{q'}}
    \|\a\|_{W^{-\alpha,p}(\cu_0)}
    \|\a^{-1}\|_{W^{-\beta,q}(\cu_0)},
  \end{equation}
  where $C(d)<\infty$.
\end{lemma}
\begin{proof}
  Let $\tilde\sigma$ be as in~\eqref{e.sigma.condition} and pick any $\varepsilon\in(0,1-\tilde\sigma)$. Choose $s,t\in(0,1]$ with
  \begin{equation*}
    s+t=1-\frac{\varepsilon}{2}
  \end{equation*}
  and so close to this split that
  \begin{equation*}
    \sigma_1\coloneqq s-\tfrac{\alpha}{2}-\tfrac{d}{2p}>0,\qquad
    \sigma_2\coloneqq t-\tfrac{\beta}{2}-\tfrac{d}{2q}>0,
  \end{equation*}
  which is possible because~\eqref{e.sigma.condition}. We now bound the two multiscale quantities appearing in the definition of $\Theta_{s,t}(\cu_0;\bacon)$. Notice that, by the definition of~$\Lambda_s$ in~\eqref{eq:besov-norms} (with $\cu_{0}$ replacing $\cu_{m}$), and~\eqref{e.a.norm.bounds}, we have
\begin{align} \label{Lam}
\Lambda_{s}(\cu_0 \,; \bacon)
& = 
\biggl( \cs_s \sum_{k=-\infty}^0 3^{sk} \max_{z \in 3^k \Z^d \cap \cu_0} \| \a(z+\cu_k;\bacon)  \|^{\nicefrac{1}{2}} \biggr)^{\! 2}
\notag \\ &
\leq
\biggl( \cs_s \sum_{k=-\infty}^0 3^{sk} \max_{z \in 3^k \Z^d \cap \cu_0} \bigl| (\a)_{z+\cu_k} \bigr|^{\nicefrac{1}{2}} \biggr)^{\! 2}
\,,
\end{align}
therefore, by~\eqref{Lam} and the triangle/monotonicity relations it is enough to control the scale-weighted maxima of the averages of $\a$ and of $\a^{-1}$. For the field $\a$ we apply Lemma~\ref{l.weak.Wsp.bounds.Besov.circ} to obtain the discrete estimate
  \begin{equation*}
    \cs_s \sum_{k=-\infty}^0 3^{s k}\max_{z\in 3^k\Z^d\cap\cu_0}\bigl|( \a)_{z+\cu_k}\bigr|^{\frac12}
    \le
    \frac{\cs_s}{\cs_{\sigma_1}}\Bigg(\frac{C}{\alpha\,p'\,(1-\alpha p')}\Bigg)^{\!\frac{1}{2p'}}\|\a\|_{W^{-\alpha,p}(\cu_0)}^{\frac12},
  \end{equation*}
  where $p'=\tfrac{p}{p-1}$ and $C=C(d)$. Squaring and recalling the definition~\eqref{eq:besov-norms} yields
  \begin{equation*}
    \Lambda_s(\cu_0;\bacon)\le
    \frac{\cs_s^2}{\cs_{\sigma_1}^2}\Bigg(\frac{C}{\alpha\,p'\,(1-\alpha p')}\Bigg)^{\!\frac{1}{p'}}\|\a\|_{W^{-\alpha,p}(\cu_0)}.
  \end{equation*}

  The same argument applied to $\a^{-1}$ (with parameters $\beta,q$ and the exponent $t$) gives
  \begin{equation*}
    \lambda_t^{-1}(\cu_0)
    \le
    \frac{\cs_t^2}{\cs_{\sigma_2}^2}\Bigg(\frac{C}{\beta\,q'\,(1-\beta q')}\Bigg)^{\!\frac{1}{q'}}\|\a^{-1}\|_{W^{-\beta,q}(\cu_0)}.
  \end{equation*}
  Multiplying the two displayed bounds and using the definition~\eqref{e.cg.ellipticity.ratio} of $\Theta_{s,t}$ yields the claimed estimate~\eqref{e.ellipticity.ratio.bound}. Finally, since the right-hand side is finite under the hypotheses, $\Theta_{s,t}(\cu_0;\baconmax)<\infty$, so $\a$ is coarse-grained elliptic with the chosen $(s,t)$.
\end{proof}

\section{Coarse-grained functional and elliptic inequalities}

\subsection{Coarse-grained Sobolev-Poincar\'e inequality}
\label{ss.functional}

In this subsection, we prove the following proposition, which states that~$H^1_\a$ embeds into~$L^{2^\ast_s}$,
where
\begin{equation*}
  2^\ast_s \coloneqq \frac{2d}{d-2(1-s)}
  \,.
\end{equation*}

\begin{proposition}
  \label{p.sobolev.poincare}
  Let $s \in (0,1]$, $n \in \Z$. There exists~$C(d) < \infty$ such that, for every~$u \in H_\a^1(\cu_n)$,
  \begin{equation*}
    3^{-n} \|u-(u)_{\cu_n}\|_{\underline{L}^{{2^\ast_s}}(\cu_n)} \leq Cs^{-1} \lambda_s^{-\nicefrac{1}{2}}(\cu_n)  \| \a^{\nicefrac12} \nabla u \|_{\underline{L}^2(\cu_n)}
    \,.
  \end{equation*}
\end{proposition}

The previous proposition is a consequence of Lemma~\ref{l.besov.poincare} and the following statement, which is essentially the same as~\cite[Lemma 2.2]{AK.HC}. While very simple, this lemma is important because it is the way the coarse-grained coefficients enter into our analysis.

\begin{lemma}[Coarse-grained Poincar\'e inequalities] \label{l.besov.fluxes.easy}
  Let~$s \in (0,1]$, $n \in \Z$,~$z \in 3^n \Zd$ and~$\bacon \in \mathcal{C}(z+\cu_n)$.
  For every~$u \in H_\a^1(z{+}\cu_n)$, we have
  \begin{equation}
    \|\nabla u\|_{\BesovDualSum{2}{1}{-s}(z+\cu_n)} \leq \frac{3^{sn}}{\cs_s} \lambda_s^{-\nicefrac{1}{2}}(z+\cu_n) \| \a^{\nicefrac12} \nabla u \|_{\underline{L}^2(z+\cu_n)}
    \,.
  \end{equation}
  For every~$u \in \bacon$,
  \begin{equation}
    \|\a \nabla u\|_{\BesovDualSum{2}{1}{-s}(z+\cu_n)} \leq \frac{3^{sn}}{\cs_s} \Lambda_s^{\nicefrac{1}{2}}(z+\cu_n; \bacon) \| \a^{\nicefrac12} \nabla u \|_{\underline{L}^2(z+\cu_n)}\,.
  \end{equation}
\end{lemma}
\begin{proof}
  Without loss of generality, we may assume that~$z=0$. Using Lemma~\ref{l.cone.homogeneity}, we get
  \begin{align*}
    \|\nabla u\|_{\BesovDualSum{2}{1}{-s}(\cu_n)} & = \sum_{k=-\infty}^n 3^{sk} \biggl( \avsum_{z \in 3^k \Z^d \cap \cu_n} |(\nabla u)_{z+\cu_k}|^2 \biggr)^{\nicefrac{1}{2}}
    \\
    & \leq \sum_{k=-\infty}^n 3^{sk} \biggl( \avsum_{z \in 3^k \Z^d \cap \cu_n} | \a_*^{-1}(z+\cu_k) | \| \a^{\nicefrac12} \nabla u \|_{\underline{L}^2(z+\cu_k)}^{2} \biggr )^{\nicefrac{1}{2}}
    \\
    & \leq \sum_{k=-\infty}^n 3^{sk} \max_{z \in 3^k \Z^d \cap \cu_n} | \a_*^{-1}(z+\cu_k) |^{\nicefrac 12} \biggl ( \avsum_{z \in 3^k \Z^d \cap \cu_n}  \| \a^{\nicefrac12} \nabla u \|_{\underline{L}^2(z+\cu_k)}^{2} \biggr )^{\nicefrac{1}{2}}
    \\
    & =
    3^{sn} \cs_s^{-1} \lambda_s^{-\nicefrac{1}{2}}(\cu_n) \| \a^{\nicefrac12} \nabla u \|_{\underline{L}^2(\cu_n)}\,.
  \end{align*}
  The proof of the second estimate is analogous.
\end{proof}

\subsection{Coarse-grained Caccioppoli inequality}
\label{ss.cg.caccioppoli}

At the heart of the analysis in this paper is the following generalization of the Caccioppoli inequality to coarse-grained elliptic coefficient fields.

\begin{proposition}[Coarse-grained Caccioppoli inequality]
  \label{p.caccioppoli}
  Let~$s,t \in (0,1)$ with~$\sigma \coloneqq 1-s-t >0$.
  Let~$u\in \mathcal{A}(\cu_0)$ and let $\Phi:\R\to\R$ be a convex function that is monotone on the range of $u$. Assume either $\lim_{t \to 0^+} \frac{\Phi(t)}{\sqrt{|\Phi'(t)|}}=0$ or $\lim_{t \to \infty} \frac{\Phi(t)}{\sqrt{|\Phi'(t)|}}=0$, and in the former (respectively, latter) case assume that there exists~$C_\Phi>0$ such that
  \begin{equation*}\Psi(t)\coloneqq\pm\int_0^t\sqrt{|\Phi'(s)|}\,ds=\pm C_\Phi\frac{\Phi(t)}{\sqrt{|\Phi'(t)|}}\,, \quad \text{resp.}\quad \Psi(t)\coloneqq\mp\int_t^\infty\sqrt{|\Phi'(s)|}\,ds=\pm C_\Phi\frac{\Phi(t)}{\sqrt{|\Phi'(t)|}}\,,
  \end{equation*}
  where the sign $\pm$ is chosen according to the monotonicity of $\Phi$. Let $\bacon(\cu_0) \in \mathcal{C}(\cu_0)$ and assume that~$\Phi(u),\Psi(u)\in\bacon(\cu_0)$. Then there exists a constant $C(d)<\infty$ such that, for every~$\nicefrac12\leq\varrho_1<\varrho_2\leq1$,
  \begin{equation}\|\a^{\nicefrac12}\nabla\Psi(u)\|_{\underline{L}^2(\varrho_1\cu_0)}
    \leq \frac{C (\frac{2}{\sigma})^{\frac{s+\sigma}{\sigma}}\Theta_{s,t}^{\frac{s}{2\sigma}}(\cu_0\,;\bacon)\Lambda_s^{\nicefrac12}(\cu_0\,;\bacon)}{C_\Phi^{\frac{s+\sigma}{\sigma}}s^{\nicefrac{s}{\sigma}}(1-s)^{\nicefrac{s}{\sigma}}(\varrho_2-\varrho_1)^{\frac{s+\sigma}{\sigma}}}
    \|\Psi(u)\|_{\underline{L}^2(\varrho_2\cu_0)}\,.
    \label{e.cg.caccioppoli}
  \end{equation}
\end{proposition}

Proposition~\ref{p.caccioppoli} is an extension of the coarse-grained Caccioppoli inequality proved for solutions in~\cite[Prop 2.5]{AK.HC}.
Before presenting its proof, we first record the following consequence.

\begin{corollary}
  Let~$s,t \in (0,1)$ with~$\sigma \coloneqq 1-s-t >0$.
  There exists a constant~$C(d)<\infty$ such that, for every~$u\in\A(\cu_0)$ and~$k>0$,
  \begin{equation} \label{e.caccioppoli.trunc}
    \|\a^{\nicefrac12}\nabla(u-k)_+\|_{\underline{L}^2(\varrho_1\cu_0)}
    \leq
    \frac{C (\frac{2}{\sigma})^{\frac{s+\sigma}{\sigma}}\Theta_{s,t}^{\frac{s}{2\sigma}}(\cu_0\,;\bacontrunc)\Lambda_s^{\nicefrac 12}(\cu_0\,;\bacontrunc)}{s^{\nicefrac{s}{\sigma}}(1-s)^{\nicefrac{s}{\sigma}}(\varrho_2-\varrho_1)^{\frac{s+\sigma}{\sigma}}}
    \|(u-k)_+\|_{\underline{L}^2(\varrho_2\cu_0)}\,.
  \end{equation}
  Moreover, if~$u\in\A(\cu_0)$ is strictly positive and bounded, and~$p\in \R \setminus \{0,1\}$, then
  \begin{equation} \label{e.caccioppoli.power}
    \|\a^{\nicefrac12}\nabla(u^{p/2})\|_{\underline{L}^2(\varrho_1\cu_0)}
    \leq
    \frac{C (\frac{2}{\sigma})^{\frac{s+\sigma}{\sigma}}\bigl(\frac{|p|}{|p-1|}\bigr)^{\frac{s+\sigma}{\sigma}}\Theta_{s,t}^{\frac{s}{2\sigma}}(\cu_0\,;\baconpow)\Lambda_s^{\nicefrac 12}(\cu_0\,;\baconpow)}{s^{\nicefrac{s}{\sigma}}(1-s)^{\nicefrac{s}{\sigma}}(\varrho_2-\varrho_1)^{\frac{s+\sigma}{\sigma}}}
    \|u^{p/2}\|_{\underline{L}^2(\varrho_2\cu_0)}.
  \end{equation}
\end{corollary}
\begin{proof}
  Let~$u\in\A(\cu_0)$.
  Taking $\Phi(u)=(u-k)_+$ in Proposition~\ref{p.caccioppoli} yields \cref{e.caccioppoli.trunc}. Assume further that~$u$  is strictly positive and bounded, and let~$p\notin\{0,1\}$. Taking $\texttt{s}(p):=\sgn(p-2)$ if $p\not =2$ and $\texttt{s}(p)=1$ if $p=2$,  $\Phi(u):=\frac{\texttt{s}(p)p^{2}u^{p-1}}{4(p-1)}$ 
  in Proposition~\ref{p.caccioppoli}, so that~$\Psi(u) = u^{\nicefrac p2}$ and~$C_\Phi = \frac{\texttt{s}(p)p}{2(p-1)}$, yields \cref{e.caccioppoli.power}.
\end{proof}

We now sketch the main idea of the proof of the coarse-grained Caccioppoli inequality. As in the classical argument, we test the equation with the product $\varphi\,\Phi(u)$, where $\varphi$ is a cutoff function. The essential difficulty is to control the cross term
\begin{equation*}
  \biggl|\fint_{\cu_0} \a \nabla \Psi(u)\cdot \Psi(u)\,\nabla \varphi \biggr| \,,
\end{equation*}
which, in the uniformly elliptic setting, one simply bounds by Cauchy--Schwarz and absorbs into the left-hand side.

Here such an absorption is not available. Instead, we estimate the cross term by the duality pairing between $\a\nabla\Psi(u)\in \BesovDualSum{2}{1}{-s}$ and
$\Psi(u)\,\nabla\varphi \in \Besov{2}{\infty}{s}$ see~\eqref{e.caccioppoli.duality}. The required Besov norms are then controlled by Lemmas~\ref{l.besov.poincare} and~\ref{l.besov.fluxes.easy}; this is precisely where the coarse-grained ellipticity constants enter.

The duality argument is carried out on an intermediate (mesoscopic) scale. Optimizing over this scale ultimately produces the pre-factor in~\eqref{e.cg.caccioppoli} and, in particular, the dependence on the coarse-grained ellipticity ratio~$\Theta_{s,t}$. Conceptually, $\Theta_{s,t}$ quantifies the price of replacing the classical pointwise (scale-independent) control ellipticity by the effective, scale-discounted Besov control used in the coarse-grained setting.

\begin{proof}[Proof of Proposition~\ref{p.caccioppoli}]
  Fix parameters $r_1,r_2 \in [\nicefrac{1}{2}, 1)$ with $r_1 < r_2$ and scale separation parameters $h,k \in \N$ with $h \geq k+4$ and $3^{-4}(r_2 - r_1) \leq 3^{-k} \leq 3^{-3}(r_2 - r_1)$. Let $\tilde{r}:=(r_{1}+r_{2})/2$, and $\varphi \in C_c^\infty(\cu_0)$ be a cutoff function such that $\indc_{r_1 \cu_{0}} \leq \varphi \leq \indc_{\tilde{r}\cu_{0}}$ and $|\nabla^j \varphi| \leq C 3^{jk} \leq C (r_2-r_1)^{-j}$ for $j=1,2$. Denote the set of subcubes which intersect the support of $\nabla \varphi$ by
  \begin{equation*}S \coloneqq \{ z \in 3^{-h} \Z^d \cap \cu_0 : (z+\cu_{-h}) \cap \supp(\nabla \varphi) \neq \emptyset \}.
  \end{equation*}
  Test the equation with $\pm\Phi(u) \varphi$ (since~$u \in \A(\cu_0)$ this is allowed by Lemma~\ref{l.product.rule}\footnote{Keep in mind that by~\eqref{e.ellipticity.monotone} and $t<1-s$ it follows that $\Theta_{s,1-s}(\cu_{0};\bacon)\le \Theta_{s,t}(\cu_{0};\bacon)$. Moreover, the sign of the test function depends on the monotonicity of $\Phi$.}) to get
  \begin{equation*}
  \pm\fint_{\cu_0} \a \nabla u \cdot \nabla (\Phi(u) \varphi) = \fint_{\cu_0} \a \nabla u \cdot \nabla u |\Phi'(u)| \varphi \pm \fint_{\cu_0} \a \nabla u \cdot \nabla \varphi \Phi(u) = 0.
  \end{equation*}
  Recalling that $\Psi'(u) = \pm\sqrt{|\Phi'(u)|}$, we have
  \begin{equation*}\fint_{\cu_0} \a \nabla u \cdot \nabla u |\Phi'(u)| \varphi
    =
    \fint_{\cu_0} \varphi \a \nabla \Psi(u) \cdot \nabla \Psi(u)
  \end{equation*}
  and, hence, using that $\Psi(u) = \pm  \frac{C_\Phi\Phi(u)}{\sqrt{|\Phi'(u)|}}$, we get
  \begin{equation}
    \fint_{\cu_0} \varphi \a \nabla \Psi(u) \cdot \nabla \Psi(u)
    \leq
    \biggl| \fint_{\cu_0} \a \nabla \Psi(u) \cdot \nabla \varphi \frac{\Phi(u)}{\sqrt{|\Phi'(u)|}} \biggr|
    =
    C_\Phi^{-1} \biggl| \fint_{\cu_0}  \a \nabla \Psi(u) \cdot \Psi(u)  \nabla \varphi \biggr|
    \,.
    \label{e.caccioppoli.premature.start}
  \end{equation}
  We write the last term as
  \begin{equation}
     \fint_{\cu_0}  \a \nabla \Psi(u) \cdot \Psi(u)  \nabla \varphi  = \avsum_{z \in 3^{-h} \Z^d \cap \cu_0} \indc_{S}(z) \fint_{z+\cu_{-h}} \Psi(u) \a \nabla \Psi(u) \cdot \nabla \varphi\,.
    \label{e.caccioppoli.start}
  \end{equation}
  For each fixed $z$ we have by duality and \cref{e.dual.norms.relation},
  \begin{align} \label{e.caccioppoli.duality}
    \lefteqn{
      \fint_{z+\cu_{-h}} \Psi(u) \a \nabla \Psi(u) \cdot \nabla \varphi
    }
    \qquad &
    \notag   \\ &
    \leq
    \biggl| \fint_{z+\cu_{-h}} (\Psi(u)-(\Psi(u))_{z+\cu_{-h}}) \a \nabla \Psi(u) \cdot \nabla \varphi \biggr|
    + |(\Psi(u))_{z+\cu_{-h}}| \biggl|  \fint_{z+\cu_{-h}} \a \nabla \Psi(u) \cdot \nabla \varphi \biggr|
    \notag   \\ &
    \leq
    3^{d+s}\|(\Psi(u) - (\Psi(u))_{z+\cu_{-h}} )\nabla \varphi\|_{\Besov{2}{\infty}{s}(z+\cu_{-h})} \|\a \nabla \Psi(u)\|_{\BesovDualSum{2}{1}{-s}(z+\cu_{-h})}
    \notag   \\ & \qquad
    +3^{d+s} |(\Psi(u))_{z+\cu_{-h}}| \|\a \nabla \Psi(u)\|_{\BesovDualSum{2}{1}{-1}(z+\cu_{-h})} \|\nabla \varphi\|_{\Besov{2}{\infty}{1}(z+\cu_{-h})}\,.
  \end{align}
  For the first term we apply Lemma~\ref{l.besov.poincare} to get
  \begin{align*}
    \|(\Psi(u) - (\Psi(u))_{z+\cu_{-h}} )\nabla \varphi\|_{\Besov{2}{\infty}{s}(z+\cu_{-h})} &
    \leq
    C 3^{-h} \| \nabla \varphi\|_{\Wul{1}{\infty}(z+\cu_{-h})} \|\nabla \Psi(u)\|_{\BesovDualSum{2}{1}{s-1}(z+\cu_{-h})}
    \\ &
    \leq
    C 3^{k} \|\nabla \Psi(u)\|_{\BesovDualSum{2}{1}{s-1}(z+\cu_{-h})}\,.
  \end{align*}
  Using our assumption that $\Psi(u) \in \bacon(\cu_0)$ (it is easy to see that it is convex since $\Phi$ is monotone), by Lemma~\ref{l.besov.fluxes.easy} we have
  \begin{equation*}\|\nabla \Psi(u)\|_{\BesovDualSum{2}{1}{s-1}(z+\cu_{-h})} \leq 3^{-(1-s)h} \cs_{1-s}^{-1} \lambda_{1-s}^{-\nicefrac{1}{2}}(z+\cu_{-h}) \|\a^{\nicefrac12} \nabla \Psi(u)\|_{\underline{L}^2(z+\cu_{-h})}.
  \end{equation*}
  Similarly, for $r \in \{1,s\}$ we have for the flux by Lemma~\ref{l.besov.fluxes.easy}
  \begin{equation*}\|\a \nabla \Psi(u)\|_{\BesovDualSum{2}{1}{-r}(z+\cu_{-h})} \leq 3^{-r h} \cs_{r}^{-1} \Lambda_{r}^{\nicefrac{1}{2}}(z+\cu_{-h}) \|\a^{\nicefrac12} \nabla \Psi(u)\|_{\underline{L}^2(z+\cu_{-h})}.
  \end{equation*}
  Assembling the above estimates together with $\|\nabla \varphi\|_{\Besov{2}{\infty}{1}(z+\cu_{-h})} \leq C 3^{k+h}$
  \begin{align*}
    \fint_{z+\cu_{-h}} \Psi(u) \a \nabla \Psi(u) \cdot \nabla \varphi
     &
    \leq
    C \frac{3^{k-h}}{\cs_{1-s} \cs_s} \Theta_{s,1-s}^{\nicefrac{1}{2}}(z+\cu_{-h})\|\a^{\nicefrac12} \nabla \Psi(u)\|_{\underline{L}^2(z+\cu_{-h})}^2
    \\ & \qquad
    +
    C 3^{k} \Lambda_1^{\nicefrac{1}{2}}(z+\cu_{-h}) \|\Psi(u)\|_{\underline{L}^2(z+\cu_{-h})} \|\a^{\nicefrac12} \nabla \Psi(u)\|_{\underline{L}^2(z+\cu_{-h})}.
  \end{align*}
  Using that~$s+t < 1$ and~$\sigma = 1-s-t$ we get from~\cref{e.ellipticity.monotone,e.Theta.scaling,e.ellipticity.scales} that
  \begin{align*}
     &
    \max_{z \in 3^{-h} \Z^d \cap \cu_0} \Theta_{s,1-s}(z+\cu_{-h}) \leq \max_{z \in 3^{-h} \Z^d \cap \cu_0} \Theta_{s,t}(z+\cu_{-h}) \leq 3^{2(s+t)h} \Theta_{s,t}(\cu_0)
    \,,
    \\ &
    \max_{z \in 3^{-h} \Z^d \cap \cu_0} \Lambda_1(z+\cu_{-h}) \leq  \max_{z \in 3^{-h} \Z^d \cap \cu_0} \Lambda_s(z+\cu_{-h}) \leq 3^{2sh} \Lambda_s(\cu_0).
  \end{align*}
  By combining the previous two displays, we get
  \begin{align*}
    \fint_{z+\cu_{-h}} \Psi(u) \a \nabla \Psi(u) \cdot \nabla \varphi
     & \leq C \frac{3^{k-\sigma h}}{\cs_{1-s} \cs_s} \Theta_{s,t}^{\nicefrac{1}{2}}(\cu_0)\|\a^{\nicefrac12} \nabla \Psi(u)\|_{\underline{L}^2(z+\cu_{-h})}^2
    \\ & \qquad
    + C 3^{k+sh} \Lambda_s^{\nicefrac{1}{2}}(\cu_0) \|\Psi(u)\|_{\underline{L}^2(z+\cu_{-h})} \|\a^{\nicefrac12} \nabla \Psi(u)\|_{\underline{L}^2(z+\cu_{-h})}.
  \end{align*}
  Notice that the bounds on the size of $h$ yield
  $$
  \bigcup_{z\in S}\bigl(z+\cu_{-h}\bigr)\subset \bigl(\tilde{r}+2\,3^{-h}\bigr)\cu_{0}\subset r_{2}\cu_{0},
  $$
  so, by summing over $z \in S$, we have, by Young's inequality,
  \begin{align*}
    \fint_{\cu_0} \a \nabla \Psi(u) \cdot \nabla \varphi \Psi(u) & \leq C \frac{3^{k-\sigma h}}{\cs_{1-s} \cs_s} \Theta_{s,t}^{\nicefrac{1}{2}}(\cu_0)\|\a^{\nicefrac12} \nabla \Psi(u)\|_{\underline{L}^2(r_2 \cu_0)}^2
    \\ & \qquad + C 3^{k+sh} \Lambda_s^{\nicefrac{1}{2}}(\cu_0) \|\Psi(u)\|_{\underline{L}^2(r_2 \cu_0)} \|\a^{\nicefrac12} \nabla \Psi(u)\|_{\underline{L}^2(r_2\cu_0)}.
  \end{align*}
  Plugging this into~\eqref{e.caccioppoli.premature.start} leads to
  \begin{align*}
    \|\a^{\nicefrac{1}{2}} \nabla \Psi(u)\|_{\underline{L}^2(r_1 \cu_0)}^2 & \leq C C_\Phi^{-1} \frac{3^{k-\sigma h}}{\cs_{1-s} \cs_s} \Theta_{s,t}^{\nicefrac{1}{2}}(\cu_0)\|\a^{\nicefrac12} \nabla \Psi(u)\|_{\underline{L}^2(r_2\cu_0)}^2
    \\ & \qquad + C C_\Phi^{-1} 3^{k+sh} \Lambda_s^{\nicefrac{1}{2}}(\cu_0) \|\Psi(u)\|_{\underline{L}^2(r_2\cu_0)} \|\a^{\nicefrac12} \nabla \Psi(u)\|_{\underline{L}^2(r_2\cu_0)}.
  \end{align*}
  Choosing $h$ such that the coefficient in front of first term in the above is less or equal than $\nicefrac{1}{4}$, specifically
  \begin{equation*}h \ge \frac{1}{\sigma} \log_3 \biggl(  \frac{4 C C_\Phi^{-1}}{\cs_{1-s} \cs_s} \Theta_{s,t}^{\nicefrac{1}{2}}(\cu_0) \biggr ) + \frac{k}{\sigma},
  \end{equation*}
  which yields
  \begin{align*}
    \lefteqn{
    \|\a^{\nicefrac{1}{2}} \nabla \Psi(u)\|_{\underline{L}^2(r_1 \cu_0)}^2
    } \ \ 
    \notag \\ &
    \leq \frac{1}{4} \|\a^{\nicefrac12} \nabla \Psi(u)\|_{\underline{L}^2(r_2 \cu_0)}^2
    + C \frac{C_\Phi^{-\frac{s+\sigma}{\sigma}}}{(\cs_s \cs_{1-s})^{\nicefrac{s}{\sigma}}} 3^{k \frac{s+\sigma}{\sigma}} \Theta_{s,t}^{\frac{s}{2\sigma}}(\cu_0)\Lambda_s^{\nicefrac{1}{2}}(\cu_0) \|\Psi(u)\|_{\underline{L}^2(r_2\cu_0)} \|\a^{\nicefrac12} \nabla \Psi(u)\|_{\underline{L}^2(r_2\cu_0)}
    \,.
  \end{align*}
  Finally, an application of Young's inequality yields  \begin{align*}
    \lefteqn{
    \|\a^{\nicefrac{1}{2}} \nabla \Psi(u)\|_{\underline{L}^2(r_1 \cu_0)}^2
    }
    \quad &
    \notag  \\ &
    \leq \frac{1}{2} \|\a^{\nicefrac12} \nabla \Psi(u)\|_{\underline{L}^2(r_2 \cu_0 )}^2
    + \frac{C}{(r_2-r_1)^{2\frac{s+\sigma}{\sigma}}}  \frac{C_\Phi^{-2\frac{s+\sigma}{\sigma}}}{(\cs_s \cs_{1-s})^{\nicefrac{2s}{\sigma}}} \Theta_{s,t}^{\nicefrac{s}{\sigma}}(\cu_0)\Lambda_s(\cu_0) \|\Psi(u)\|_{\underline{L}^2(r_2\cu_0)}^2
    \,.
  \end{align*}
  We next apply Lemma~\ref{l.iteration} to the function
  \begin{equation*}G(r) \coloneqq \|\a^{\nicefrac{1}{2}} \nabla \Psi(u)\|_{\underline{L}^2(r \cu_0)}^2
    \,.
  \end{equation*}
  Fix the maximal scale to be $\varrho_2 \in [\nicefrac{1}{2},1)$ (as in the statement) and note that for every $r_1, r_2 \in [\nicefrac{1}{2},\varrho_2)$ with $r_1 < r_2$ we have
  \begin{equation*}G(r_1) \leq \frac{1}{2} G(r_2) + \frac{A}{(r_2-r_1)^\xi}
  \end{equation*}
  where
  \begin{equation*}A \coloneqq C \frac{C_\Phi^{-2\frac{s+\sigma}{\sigma}}}{(\cs_s \cs_{1-s})^{\nicefrac{2s}{\sigma}}} \Theta_{s,t}^{\nicefrac{s}{\sigma}}(\cu_0)\Lambda_s(\cu_0) \|\Psi(u)\|_{\underline{L}^2(\varrho_2 \cu_0)}^2
    \qquad
    \mbox{and} \qquad
    \xi \coloneqq 2 \frac{s+\sigma}{\sigma}
    \,.
  \end{equation*}
  The application of the lemma yields a constant~$C(d)<\infty$ such that, for every $\varrho_1 \in [\nicefrac{1}{2},\varrho_2)$,
  \begin{equation*}\|\a^{\nicefrac{1}{2}} \nabla \Psi(u)\|_{\underline{L}^2(\varrho_1 \cu_0)}^2
    \leq
    \frac{C (2\xi)^{\xi} }{(\varrho_2-\varrho_1)^{\xi}}
    \frac{C_\Phi^{-\xi}}{(\cs_s \cs_{1-s})^{\nicefrac{2s}{\sigma}}} \Theta_{s,t}^{\nicefrac{s}{\sigma}}(\cu_0)\Lambda_s(\cu_0) \|\Psi(u)\|_{\underline{L}^2(\varrho_2 \cu_0)}^2\,.
  \end{equation*}
  Noting that~$\cs_s \approx s$, we obtain the desired result.
\end{proof}

\subsection{Reverse H\"older's inequality}

By applying the coarse-grained Caccioppoli and the coarse-grained Sobolev-Poincar\'e inequalities, we obtain a reverse H\"older inequality.

\begin{proposition}[Reverse H\"older inequality]
  \label{p.reverse.holder}
  Let~$s,t \in (0,1)$ with $\sigma\coloneqq 1-s-t > 0$.
  There exists a constant $C(d)<\infty$ such that, for every~$u$,~$\bacon$,~$\Phi$ and~$\Psi$ as in Proposition~\ref{p.caccioppoli}, and $\nicefrac12 \leq \varrho_1 < \varrho_2 \leq 1$, we have, with~$2^\ast_t \coloneqq \frac{2d}{d-2(1-t)}$,
  \begin{equation}
    \label{e.reverse.holder}
    \|\Psi(u)\|_{\underline{L}^{2^\ast_t}(\varrho_1 \cu_0)}
    \leq
    \frac{C}{(\varrho_2 - \varrho_1)^{\nicefrac d2}} \biggl(
    1+ \biggl(
    \frac{C \Theta_{s,t}^{\nicefrac{1}{2}}(\cu_0\,;\bacon)}{ C_\Phi \sigma t(\varrho_2-\varrho_1) }
    \biggr)^{\!\frac{s+\sigma}{\sigma}} \biggr)
    \|\Psi(u)\|_{\underline{L}^2(\varrho_2 \cu_0)}
    \,.
  \end{equation}
\end{proposition}
\begin{proof}
  Fix~$\varrho_1, \varrho_2 \in [\nicefrac 12,1]$ satisfying $\varrho_1 < \varrho_2$. Applying Propositions~\ref{p.sobolev.poincare} and~\ref{p.caccioppoli} to~$\Psi(u)$, we obtain, for every~$h \in \N$ and~$z \in 3^{-h} \Zd \cap \cu_0$ such that~$z+\cu_{-h} \subseteq \varrho_2 \cu_0$, that
  \begin{equation*}\|\a^{\nicefrac12}\nabla\Psi(u)\|_{\underline{L}^2(z+\cu_{-h-1})}
    \leq \frac{C (\frac{2}{\sigma})^{\frac{s+\sigma}{\sigma}}\Theta_{s,t}^{\frac{s}{2\sigma}}(z{+}\cu_{-h})\Lambda_s^{\nicefrac12}(z{+}\cu_{-h})}{C_\Phi^{\frac{s+\sigma}{\sigma}}s^{\nicefrac{s}{\sigma}}(1-s)^{\nicefrac{s}{\sigma}}}
    3^{h} \|\Psi(u)\|_{\underline{L}^2(z+\cu_{-h})}\,.
  \end{equation*}
  and
  \begin{equation*}
    3^{h} \| \Psi(u) -( \Psi(u) )_{z+\cu_{-h-1}}\|_{\underline{L}^{2^\ast_t}(z+\cu_{-h-1})} \leq \frac{C}{t} \lambda_t^{-\nicefrac{1}{2}}(z+\cu_{-h-1})  \| \a^{\nicefrac12} \nabla \Psi(u) \|_{\underline{L}^2(z+\cu_{-h-1})}\,.
  \end{equation*}
  By~\eqref{e.ellipticity.scales} and~\eqref{e.Theta.scaling}, we have
  \begin{equation*}
    \Theta_{s,t}^{\frac{s}{2\sigma}}(z{+}\cu_{-h}\,;\bacon)\Lambda_s^{\nicefrac12}(z{+}\cu_{-h}\,;\bacon) \lambda_t^{-\nicefrac{1}{2}}(z{+}\cu_{-h-1})\leq C 3^{ \frac{h(s+\sigma)(s+t)}{\sigma}} \Theta_{s,t}^{\frac{s+\sigma}{2\sigma}}(\cu_0)\,
    \leq
    C 3^{h \frac{s+\sigma}{\sigma}} \Theta_{s,t}^{\frac{s+\sigma}{2\sigma}}(\cu_0)
    \,.
  \end{equation*}
  Thus, we obtain
  \begin{equation*}
    \| \Psi(u) -( \Psi(u) )_{z+\cu_{-h-1}}\|_{\underline{L}^{2^\ast_t}(z+\cu_{-h-1})}
    \leq
    \frac{1}{ts^{\nicefrac{s}{\sigma}}(1-s)^{\nicefrac {s}{\sigma}}} \left( \frac{C 3^{h}  \Theta_{s,t}^{\nicefrac12} (\cu_0) }{\sigma C_{\Phi}}\right)^{\frac{s+\sigma}{\sigma}}
    \|\Psi(u)\|_{\underline{L}^2(z+\cu_{-h})} \,.
  \end{equation*}
  Using~$1- s = t + \sigma \geq t$ and
  \begin{equation*}
    \bigl| ( \Psi(u) )_{z+\cu_{-h-1}} \bigr|
    \leq
    3^{\nicefrac d2} \|\Psi(u)\|_{\underline{L}^2(z+\cu_{-h})}
    \leq
    C3^{h \frac d2} \|\Psi(u)\|_{\underline{L}^2(\varrho_2 \cu_{0})}
    \,,
  \end{equation*}
  we then obtain~\eqref{e.reverse.holder} by the triangle inequality and a covering argument by taking~$h$ so large that~$3^{-4}(\varrho_2 - \varrho_1)< 3^{-h} \leq 3^{-3}(\varrho_2 - \varrho_1)$.
\end{proof}

Using \cref{e.caccioppoli.power} instead for positive and bounded $u \in \A(\cu_0)$ and $p \in \R \setminus \{0,1\}$ we have
\begin{equation}\label{e.onestep.moser}
  \|u^{\nicefrac{p}{2}} \|_{\underline{L}^{2^\ast_t}(\varrho_1 \cu_0)}
  \leq
  \frac{C}{(\varrho_2 - \varrho_1)^{\nicefrac d2}} \biggl(
  1+ \biggl(
  \frac{C |p| \Theta_{s,t}^{\nicefrac{1}{2}}(\cu_0\,;\baconpow)}{|p-1| \sigma t (\varrho_2-\varrho_1) }
  \biggr)^{\! \frac{s+\sigma}{\sigma}} \biggr)
  \|u^{\nicefrac{p}{2}}\|_{\underline{L}^2(\varrho_2 \cu_0)}.
\end{equation}
The fact that $u^p \in \bacon^{\scriptscriptstyle{\mathrm{pow}}}_{\cu_0}$ follows from Lemma~\ref{l.chain.rule}.
For truncations $k > 0$, we have
\begin{equation}
  \label{e.onestep.degiorgi}
  \|(u-k)_+ \|_{\underline{L}^{2^\ast_t}(\varrho_1 \cu_0)}
  \leq
  \frac{C}{(\varrho_2 - \varrho_1)^{\nicefrac d2}} \biggl(
  \frac{C \Theta_{s,t}^{\nicefrac{1}{2}}(\cu_0\,;\bacontrunc)}{\sigma t(\varrho_2-\varrho_1 )}
  \biggr)^{\! \frac{s+\sigma}{\sigma}}
  \|(u-k)_+\|_{\underline{L}^2(\varrho_2 \cu_0)}
\end{equation}
The fact that $(u-k)_+ \in H_\a^{1}(\cu_0)$ follows from the remark after Lemma~\ref{l.chain.rule}.

\subsection{Logarithmic estimates}

We next present the logarithmic-type Caccioppoli inequality.

\begin{proposition} \label{c.log.caccioppoli}
  Let~$s,t \in (0,1)$ with $\sigma\coloneqq 1-s-t > 0$.
  There exists a constant $C(d)<\infty$ such that, if $u \in \A(\cu_0)$ is strictly positive and bounded, then
  \begin{equation}
    \|\a^{\nicefrac{1}{2}} \nabla \log u\|_{\underline{L}^2(\frac{1}{2} \cu_0)}^2 \leq C \Lambda_s(\cu_0\,;\baconpow) \,.
    \label{e.log.energy.est}
  \end{equation}
  Furthermore, there exists a constant $c(d) \in(0,1)$ such that if we define
  \begin{equation}
    w \coloneqq \exp \left ( - (\log u^{p_{*}} )_{\frac{1}{2} \cu_0} \right ) u^{p_{*}}
    \quad \mbox{with} \quad p_{*} \coloneqq c t \Theta_{s,t}^{-\nicefrac{1}{2}}(\cu_0\,;\baconpow)\in (0,\nicefrac{1}{2}]\,,
    \label{e.def.w.log}
  \end{equation}
  then
  \begin{equation}
    \left\| \log w \right\|_{\underline{L}^{2^*_t}(\frac{1}{2} \cu_0)} \leq 1\,.
    \label{e.log.w.estimate}
  \end{equation}
\end{proposition}
\begin{proof}
  We test the equation with $\Phi(u) \varphi$, $\Phi(u) = u^{-1}$, with $\Psi(u) = \log u$. Note that~$\Phi(u),-\Psi(u) \in \baconpow(\cu_0)$. We obtain
  \begin{equation*}\fint_{\cu_0} \varphi \a \nabla \Psi(u) \cdot \nabla \Psi(u) = \fint_{\cu_0} \a \nabla \Psi(u) \cdot \nabla \varphi\,.
  \end{equation*}
  Arguing as in the proof of Proposition~\ref{p.caccioppoli}, in place of~\cref{e.caccioppoli.duality} we obtain
  \begin{equation*}\fint_{z+\cu_{-h}} \a \nabla \Psi(u) \cdot \nabla \varphi
    \leq  C \|\a \nabla \Psi(u)\|_{\BesovDualSum{2}{1}{-1}(z+\cu_{-h})} \|\nabla \varphi\|_{\Besov{2}{\infty}{1}(z+\cu_{-h})}\,,
  \end{equation*}
  from which we obtain the estimate
  \begin{eqnarray*}
    \|\a^{\nicefrac{1}{2}} \nabla \Psi(u)\|_{\underline{L}^2(r_1 \cu_0)}^2
     & \leq&
    C 3^{k+sh} \Lambda_s^{\nicefrac{1}{2}}(\cu_{0};\baconpow) \|\a^{\nicefrac12} \nabla \Psi(u)\|_{\underline{L}^2(r_2 \cu_0)}\nonumber \\
    &\leq&
    \frac{1}{2} \|\a^{\nicefrac12} \nabla \Psi(u)\|_{\underline{L}^2(r_2 \cu_0)}^2 + C 3^{2(k+sh)} \Lambda_s(\cu_0;\baconpow)\,.
  \end{eqnarray*}
  We next select $h \coloneqq k+4$ so that~$3^{2(k+sh)} \leq C (r_2-r_1)^{-4}$. Performing an iteration as in the proof of Proposition~\ref{p.caccioppoli}, we obtain~\eqref{e.log.energy.est}.

  \smallskip

  Turning to the proof of the second statement, we let~$w$ and~$p_{*}$ be as in~\eqref{e.def.w.log}. From the first part of the proof, we have
  \begin{equation*}\|\a^{\nicefrac{1}{2}} \nabla \log w\|_{\underline{L}^2(\frac{1}{2} \cu_0)} = p_{*} \|\a^{\nicefrac{1}{2}} \nabla \log u\|_{\underline{L}^2(\frac{1}{2} \cu_0)} \leq C p_{*} \Lambda_s^{\nicefrac{1}{2}}(\cu_0;\baconpow)
    \,.
  \end{equation*}
  Observe that~$(\log w)_{\frac{1}{2} \cu_0} = 0$ by construction. Applying Lemma~\ref{l.besov.poincare} to $v = \log w$ and using Lemma~\ref{l.besov.fluxes.easy} (note that~$v \in \baconpow(\cu_0)$), we obtain
  \begin{equation*}\|v\|_{\underline{L}^{2^*_t}(\frac{1}{2} \cu_0)} \leq C\|\nabla v\|_{\BesovDualSum{2}{1}{-t}(\frac{1}{2} \cu_0)} \leq C t^{-1}\lambda_t^{-\nicefrac{1}{2}}(\cu_0) \|\a^{\nicefrac{1}{2}} \nabla v\|_{\underline{L}^2(\frac{1}{2} \cu_0)} \leq C t^{-1}p_{*} \Theta_{s,t}^{\nicefrac{1}{2}}(\cu_0)\,.
  \end{equation*}
  Using the definition of~$p_{*}$ yields~\eqref{e.log.w.estimate} and completes the proof.
\end{proof}

\section{Local boundedness}

In this section, we present the proof of the local boundedness estimate (Theorem~\ref{t.local.boundedness.intro}). In fact, we will prove the following more general statement. The proof is a straightforward adaptation of the classical De Giorgi iteration, using the coarse-grained Caccioppoli and Sobolev-Poincar\'e inequalities developed in the previous section.

\begin{theorem}[Local boundedness]\label{t.local.bound}
  Let~$s,t \in (0,1)$ with $\sigma\coloneqq 1-s-t > 0$.
  There exists~$C(d)<\infty$ such that, for every~$u \in \A(\cu_0)$, we have
  \begin{equation*}
    \sup_{\frac{1}{2}\cu_0} u \leq
    \Bigl(
    \tfrac{C^{\frac{1}{1-t}}}{ \sigma^{2} t^{2}} \Theta_{s,t}(\cu_0\,;\bacontrunc)
    \Bigr)^{\! \frac{d}{4 \sigma}} \|u_+\|_{\underline{L}^2(\cu_0)}\,,
  \end{equation*}
  where $u_+ = \max\{u, 0\}$ denotes the positive part of $u$. If $u \in \A_-(\cu_0)$, then the same estimate holds with respect to $\baconsub$ in place of~$\bacontrunc$.
\end{theorem}

\begin{remark}
  The exponent of the coarse-grained ellipticity ratio
  $\Theta_{s,t}(\cu_0;\bacontrunc(\cu_0))$ in Theorem~\ref{t.local.bound}
  is $\nicefrac{d}{4\sigma}$. In the uniformly elliptic setting one may
  take $s=t=0$, hence $\sigma=1$, and this exponent reduces to
  $\nicefrac{d}{4}$, in agreement with the classical estimate of
  Trudinger~\cite{Trudinger}.
\end{remark}

\begin{proof}
  Denoting $u_k = (u-k)_+$ and construct a sequence of $\varrho_m = \frac{1}{2} + \frac{1}{2^{m+1}}$, we denote the sequence of cubes $\hat \cu_m \coloneqq \varrho_m \cu_0$.
  Now, for $\delta_m \in (0,2^{-d+1}]$ to be fixed, set
  \begin{equation*}k_{m+1} \coloneqq k_m + \delta_m^{-1} \|u_{k_m}\|_{\underline{L}^2(\hat \cu_m)}
    \quad \text{and} \quad
    \tilde k_m \coloneqq \frac{1}{2} (k_{m} + k_{m+1})\,,
  \end{equation*}
  with $k_0 = 0$.
  Applying \cref{e.onestep.degiorgi} to $u_{\tilde k_m}$ in the cubes $\hat \cu_{m+1} \subseteq \hat \cu_m$ we have
  \begin{equation}\label{e.onestep.degiorgi.LB}
    \|u_{\tilde k_m} \|_{\underline{L}^{2^\ast_t}(\hat \cu_{m+1})}
    \leq
    2^{(\frac d2 + \frac{s+\sigma}{\sigma})   m}
    \bigl(
    C \tfrac{1}{\sigma t} \Theta_{s,t}^{\nicefrac{1}{2}}(\cu_0\,;\bacontrunc)
    \bigr)^{ \frac{s+\sigma}{\sigma}}
    \|u_{\tilde k_m}\|_{\underline{L}^2(\hat \cu_m)},
  \end{equation}
  where $2^\ast_t \coloneqq \frac{2d}{d-2(1-t)}$ and~$\sigma = 1 - s - t$.
  By Chebyshev's inequality, we get
  \begin{equation*}\frac{|\hat \cu_{m+1} \cap \{ u > \tilde k_m \}|}{|\hat \cu_{m+1}|} \leq \frac{4 \|u_{k_m}\|_{\underline{L}^{2}(\hat \cu_{m+1})}^2}{(k_{m+1}-k_m)^2} \leq
    C \delta_m^2\,.
  \end{equation*}
  Applying Hölder's inequality together with the above two displays gives us
  \begin{align*}
    \|u_{k_{m+1}}\|_{\underline{L}^2(\hat \cu_{m+1})} & \leq
    \|u_{\tilde k_m}\|_{\underline{L}^{2^\ast_t}(\hat \cu_{m+1})}
    \biggl(
    \frac{|\hat \cu_{m+1} \cap \{u > \tilde k_m\}|}{|\hat \cu_{m+1}|} \biggr)^{\!\frac{1-t}{d}} \\
    & \leq
    2^{(\frac d2 + \frac{s+\sigma}{\sigma})   m}
    \bigl(
    \tfrac{C}{\sigma t}  \Theta_{s,t}^{\nicefrac{1}{2}}(\cu_0\,;\bacontrunc)
    \bigr)^{\! \frac{s+\sigma}{\sigma}}
    \|u_{\tilde k_m}\|_{\underline{L}^2(\hat \cu_m)}
    \biggl( \frac{|\hat \cu_{m+1} \cap \{u > \tilde k_m\}|}{|\hat \cu_{m+1}|} \biggr)^{\!\frac{1-t}{d}}
    \\
    &
    \leq
    2^{(\frac d2 + \frac{s+\sigma}{\sigma})   m}
    \bigl(
    \tfrac{C}{\sigma t}  \Theta_{s,t}^{\nicefrac{1}{2}}(\cu_0\,;\bacontrunc)
    \bigr)^{\! \frac{s+\sigma}{\sigma}}
    (k_{m+1}-k_m) \delta_m^{1 + \frac{2(1-t)}{d}}
    \,.
  \end{align*}
  We now choose
  \begin{equation*}
    \delta_{m+1} \coloneqq   2^{2+(\frac d2 + \frac{s+\sigma}{\sigma})   m}
    \bigl(
    \tfrac{C}{\sigma t}  \Theta_{s,t}^{\nicefrac{1}{2}}(\cu_0\,;\bacontrunc)
    \bigr)^{ \frac{s+\sigma}{\sigma}} \delta_m^{1 + \frac{2(1-t)}{d}}\,,
  \end{equation*}
  which leads to
  \begin{equation}
    \label{e.k.m.bound}
    k_{m+2} - k_{m+1} =
    \delta_{m+1}^{-1} \|u_{k_{m+1}}\|_{\underline{L}^2(\hat \cu_{m+1})} \leq \frac{1}{2}(k_{m+1}-k_m).
  \end{equation}
  We show that if~$\delta_0$ is small enough, then~$\delta_m \to 0$. To this end, define
  \begin{equation*}
    K\coloneqq 4
    \bigl(
    \tfrac{C}{\sigma t}  \Theta_{s,t}^{\nicefrac{1}{2}}(\cu_0\,;\bacontrunc)
    \bigr)^{\! \frac{s+\sigma}{\sigma}}\,, \quad
    \xi \coloneqq \frac d2 + \frac{s+\sigma}{\sigma}
    \quad \mbox{and} \quad
    \varepsilon \coloneqq
    \frac{2(1-t)}{d}
    \,.
  \end{equation*}
  We then have that
  \begin{equation*}
    \delta_{m+1} =   2^{\xi m} K \delta_m^{1 + \ep}
    \,.
  \end{equation*}
  By a direct computation, we deduce the following iterative implication:
  \begin{equation*}\delta_m \leq 2^{- \xi \frac{m}{\varepsilon} - \frac{\xi}{\varepsilon^2}}  K^{-\nicefrac{1}{\varepsilon}} \quad \implies \quad \delta_{m+1} \leq 2^{- \xi \frac{m+1}{\varepsilon} - \frac{\xi}{\varepsilon^2}}  K^{-\nicefrac{1}{\varepsilon}} \,.
  \end{equation*}
  Therefore, if $\delta_0 \leq K^{-\nicefrac{1}{\varepsilon}}2^{- \nicefrac{\xi}{\varepsilon^2}}$, then $\delta_m$ forms a decreasing sequence bounded by $C_0^{-\nicefrac{1}{\varepsilon}} 2^{- \nicefrac{m}{\varepsilon}-\nicefrac{\xi}{\varepsilon^2}}$.

  \smallskip

  Next, by summing~\eqref{e.k.m.bound} we obtain
  \begin{equation*}k_m - k_0 = \sum_{j=0}^{m-1} (k_{j+1} - k_{j}) \leq 2(k_{1}-k_0) = 2 K^{\nicefrac{1}{\varepsilon}}2^{\nicefrac{\xi}{\varepsilon^2}} \|(u-k_0)_+\|_{\underline{L}^2(\cu_0)}\,.
  \end{equation*}
  Thus,~$k_m$ is uniformly bounded and the limit $k_\infty$ exists.   In particular, since~$\delta_m \to 0$, we have
  \begin{equation*}\lim_{m \to \infty} \delta_m^{-1} \|u_{k_m}\|_{\underline{L}^2(\hat \cu_m)} = 0 \implies
    u \leq k_\infty \quad \mbox{a.e.~in } \hat \cu_\infty = \tfrac{1}{2}\cu_0\,.
  \end{equation*}
  This gives the desired result since~$\frac{s+\sigma}{1-t} = 1$.
\end{proof}

\section{The Harnack inequality}
\label{s.Harnack}

\subsection{Moser iteration}

Suppose that $u \in \mathcal{A}(\cu_0)$ is bounded and positive in~$\cu_0$. Then we have by~\eqref{e.onestep.moser} that, for every~$\nicefrac 12 \leq \varrho_1 < \varrho_2 \leq 1$ and~$p\in \R\setminus \{0,1\}$,
\begin{equation}\label{e.onestep.moser.again}
  \|u^{\nicefrac{p}{2}} \|_{\underline{L}^{2^\ast_t}(\varrho_1 \cu_0)}
  \leq
  \frac{C}{(\varrho_2 - \varrho_1)^{\nicefrac d2}} \biggl(
  1+ \biggl(
  \frac{C |p| \Theta_{s,t}^{\nicefrac{1}{2}}(\cu_0\,;\baconpow)}{|p-1| \sigma t (\varrho_2-\varrho_1) }
  \biggr)^{\! \frac{s+\sigma}{\sigma}} \biggr)
  \|u^{\nicefrac{p}{2}}\|_{\underline{L}^2(\varrho_2 \cu_0)}.
\end{equation}
Renaming $v = u^{\sgn(p)}$ and rewriting the above gives, for~$\kappa \coloneqq \frac{d}{d-2(1-t)}$,
\begin{equation} \label{e.moser.step}
  \|v \|^{|p|}_{\underline{L}^{\kappa |p|}(\varrho_1 \cu_0)}
  \leq \frac{C}{(\varrho_2 - \varrho_1)^{d}} \biggl(
  1+ \biggl(
  \frac{C |p| \Theta_{s,t}^{\nicefrac{1}{2}}(\cu_0\,;\baconpow)}{|p-1| \sigma t (\varrho_2-\varrho_1) }
  \biggr)^{\! \frac{s+\sigma}{\sigma}} \biggr)^{\!2}  \|v \|^{|p|}_{\underline{L}^{|p|}(\varrho_2 \cu_0)},
\end{equation}
this is the desired single step inequality for Moser iteration.

\begin{proposition} \label{p.moser.iteration.all.p}
  Let~$s,t \in (0,1)$ with~$\sigma\coloneqq 1-s-t > 0$.
  There exists~$C(d)<\infty$ such that, for every positive and bounded~$u \in \mathcal{A}(\cu_0)$, $\nicefrac12 \leq \varrho_1 < \varrho_2 \leq 1$ and~$p\in \R \setminus [0,1]$, we have
  \begin{equation}
    \label{e.infty.p}
    \|u^{p}\|_{{L}^{\infty}(\varrho_1\cu_0)}
    \leq
    \biggl( \frac{C^{\frac{1}{1-t}}}{(\varrho_2-\varrho_1)^{d}} \left (1 + \frac{|p|}{\sigma t|p-1|} \Theta_{s,t}^{\nicefrac{1}{2}}(\cu_0\,;\baconpow ) \right )  \biggr)^{\!\nicefrac{d}{\sigma}} \|u^{p}\|_{{L}^{1}(\varrho_2 \cu_0)} \,.
  \end{equation}
  Moreover, for every $p,q \in (0,1)$ with $p < q$, and~$u,\varrho_1,\varrho_2$ as above, we have for~$\kappa \coloneqq \frac{d}{d-2(1-t)}$,
  \begin{equation}
    \label{e.small.p}
    \|u^{q\kappa}\|_{\underline{L}^{1}(\varrho_1 \cu_0)}^{\frac{p}{q \kappa}}
    \leq
    \biggl( \frac{C^{\frac{1}{1-t}}}{(\varrho_2-\varrho_1)^{d}} \left (1 + \frac{p}{\sigma t(1-q)} \Theta_{s,t}^{\nicefrac{1}{2}}(\cu_0\,;\baconpow ) \right )  \biggr)^{\!\nicefrac{d}{\sigma}}
    \|u^p\|_{\underline{L}^{1}(\varrho_2 \cu_0)}\,.
  \end{equation}
\end{proposition}
The proof of the above proposition relies on already qualitatively knowing the local boundedness of $u$, which we have from Theorem~\ref{t.local.bound}. We also note that~\eqref{e.infty.p} applied to~$p = 2$ coincides with the local boundedness estimate as in Theorem~\ref{t.local.bound}, with the same power of $\Theta_{s,t}$, although with respect to $\bacon^{\scriptscriptstyle{\mathrm{pow}}}(\cu_0)$ instead of $\bacontrunc(\cu_0)$.
\begin{proof}
  Define the sequence of exponents $p_m$ and radii $r_m$ as
  \begin{equation}
    \label{e.p.m.def}
    p_{m+1} \coloneqq \frac{d}{d-2(1-t)} p_m \eqqcolon  \kappa p_m \,, \quad p_0 = p\,,
    \quad \mbox{and} \quad
    r_m = \varrho_1 + \frac{\varrho_2-\varrho_1}{2^m}\,,
  \end{equation}
  with any initial exponent $p_0 \in \R \setminus [0,1]$,
  and the sequence of cubes $\hat \cu_m \coloneqq r_m \cu_0$. Then applying \cref{e.moser.step} with $p = p_m$ and $\varrho_1 = r_{m+1}$, $\varrho_2 = r_m$, we obtain
  \begin{equation*}\|v \|_{\underline{L}^{|p_{m+1}|}(\hat \cu_{m+1})}
    \leq \frac{(C2^{md})^{\frac{1}{|p_{m}|}}}{(\varrho_2 - \varrho_1)^{\frac{d}{|p_m|}} } \biggl(
    1+ \biggl(
    \frac{C 2^m  |p_m| \Theta_{s,t}^{\nicefrac{1}{2}}(\cu_0;\baconpow)}{|p_m-1| \sigma t (\varrho_2 - \varrho_1)}
    \biggr)^{\! \frac{s+\sigma}{\sigma}} \biggr)^{\!\frac{2}{|p_m|}} \| v \|_{\underline{L}^{|p_m|}(\hat \cu_m)}
  \end{equation*}
  which we can iterate to get
  \begin{equation*}\|v\|_{\underline{L}^{\infty}(\varrho_{1} \cu_{0})}
    \leq
    C   \|v\|_{\underline{L}^{|p|}(\varrho_{2}\cu_{0})} \prod_{m=0}^{\infty} \frac{2^{\frac{md}{|p_m|}}}{(\varrho_2 - \varrho_1)^{\frac{d}{|p_m|}} }
    \biggl(
    1+ \biggl( \frac{C 2^m  |p_m| \Theta_{s,t}^{\nicefrac{1}{2}}(\cu_0;\baconpow)}{|p_m-1| \sigma t(\varrho_2 - \varrho_1)}
    \biggr)^{\! \frac{s+\sigma}{\sigma}} \biggr)^{\!\frac{2}{|p_m|}}
  \end{equation*}
  Denote, for short,~$K \coloneqq C (\sigma t)^{-1} (\varrho_2 - \varrho_1)^{-1}\Theta_{s,t}^{\nicefrac{1}{2}}(\cu_0\,;\baconpow)$. Taking logarithm of the product above, we deduce that
  \begin{align*}
    \lefteqn{
      \sum_{j=0}^\infty
      \biggl( \frac{d j}{|p_j|} \log 2 + \frac{d}{|p_j|} \bigl| \log ( \varrho_2-\varrho_1 ) \bigr|  + \frac{2}{|p_j|} \log \biggl( 1 + \Bigl ( \frac{|p_j|2^j}{|p_j-1|}  K \Bigr)^{\frac{s+\sigma}{\sigma}} \biggr) \biggr)
    } \quad &
    \notag    \\ &
    \leq
    \frac{1}{|p_0|} \biggl( \frac{\log C}{(1-\kappa^{-1})^2}
    +
    \frac{d}{(1-\kappa^{-1})} \bigl| \log \left( \varrho_2-\varrho_1 \right) \bigr|
    +
    \frac{2(s+\sigma)}{\sigma(1-\kappa^{-1})} \log \Bigl(1+\frac{|p_{0}|K}{|p_{0}-1|} \Bigr)
    \biggr)
    \notag    \\ &
    \leq
    \frac{1}{|p_0|} \biggl( \frac{\log C}{(s+\sigma)\sigma}
    +
    \frac{d^2} {2(s+\sigma)} \bigl| \log ( \varrho_2-\varrho_1 ) \bigr|
    + \frac{d}{\sigma} \log \Bigl(1+\frac{C|p_{0}| K}{|p_0-1|} \Bigr)  \biggr)
    \,,
  \end{align*}
  where in the last line we used~$1-\kappa^{-1}  = \frac{2(s+\sigma)}{d}$. Taking exponential of the above and combining the previous two displays yields the result.

  \smallskip

  The proof of~\eqref{e.small.p} is analogous, and we omit it.
\end{proof}

\subsection{The crossover lemma}
\label{ss.crossover}

We have two crossovers to perform, for positive exponent we cross at $p=1$ and for negative exponents we cross at $p=0$. The first one is easy, from~\eqref{e.small.p} we can find $q < 1$ such that $\frac{q d}{d-2(1-t)} > 1$, then use this as $p_0$ for~\eqref{e.infty.p} to get that we can pass from any $p \in (0,1)$ to any $q \geq p$. The second crossover is more delicate and requires the use of the Bombieri lemma.

\begin{lemma} [Crossover lemma] \label{l.crossover}
  Let~$s,t \in (0,1)$ with~$\sigma\coloneqq 1-s-t > 0$, and let~$u$ be a nonnegative weak solution of~$-\nabla \cdot \a\nabla u  = 0$ in~$\cu_0$. There exist constants~$C = C(d)<\infty$ and~$c = c(d) \in (0,\nicefrac{1}{2}]$
  such that
  \begin{equation*}\|u^{p_\ast}\|_{\underline{L}^1(\frac{1}{4}\cu_0)}
    \|u^{-p_\ast}\|_{\underline{L}^{1}(\frac{1}{4}\cu_0)}
    \leq
    \exp\bigl( C^{\frac{1}{\sigma(1-t)}} (\tfrac{C}{\sigma})^{\nicefrac{C}{\sigma}} \bigr)
    \quad \mbox{with} \quad
    p_\ast \coloneqq c t \Theta_{s,t}^{-\nicefrac{1}{2}}(\cu_0\,;\baconpow(\cu_0))
    \,.
  \end{equation*}
\end{lemma}
\begin{proof}
  By Proposition~\ref{c.log.caccioppoli}, we have that
  \begin{equation*}w \coloneqq \exp \left ( - (\log u^{p_\ast} )_{\frac{1}{2} \cu_0} \right ) u^{p_\ast}
    \quad \mbox{with} \quad p_\ast \coloneqq c t \Theta_{s,t}^{-\nicefrac{1}{2}}(\cu_0\,;\baconpow) \in (0,\nicefrac12 ]\,,
  \end{equation*}
  satisfies
  \begin{equation}\label{e.bombieri.condition}
    \left\|\log w\right\|_{L^1(\frac{1}{2}\cu_0)} \leq 1.
  \end{equation}
  Using~\eqref{e.small.p}, we obtain
  \begin{equation}
    \label{e.small.p.applied}
    \|u\|_{\underline{L}^{\frac{q d}{d-2(1-t)}}(\varrho_1 \cu_0)}^p
    \leq
    \biggl( \frac{C^{\frac{1}{1-t}}}{(\varrho_2-\varrho_1)^{d}} \biggl (1 + \frac{p}{\sigma t(1-q)} \Theta_{s,t}^{\nicefrac{1}{2}}(\cu_0\,;\baconpow ) \biggr )  \biggr)^{\!\nicefrac{d}{\sigma}}
    \|u^{p}\|_{\underline{L}^{1}(\varrho_2 \cu_0)}\,.
  \end{equation}
  Applying the above with~$ p = \hat{p} p_\ast$ and $q = \hat q p_\ast$ with~$0<\hat{p} \leq \hat{q} \leq 1$ gives us\begin{equation*}\|w \|_{\underline{L}^{\hat{q}}(\varrho_1 \cu_0)}
    \leq
    \biggl( \frac{C^{\frac{1}{1-t}}}{\sigma (\varrho_2-\varrho_1)^{d}}  \biggr)^{\!\frac{d}{\hat{p} \sigma}}
    \|w\|_{\underline{L}^{\hat{p}}(\varrho_2 \cu_0)}
  \end{equation*}
  Thus, we can apply Lemma~\ref{l.bombieri} to~$w$ and obtain that there exists a constant $C(d)<\infty$ such that
  \begin{equation*}\|w\|_{\underline{L}^{1}(\frac{1}{4} \cu_0)}
    \leq
    \exp\bigl ( C^{\frac{1}{\sigma(1-t)}} (\tfrac{C}{\sigma})^{\nicefrac{C}{\sigma}} \bigr )
    \,.
  \end{equation*}
  Repeating the same argument for $u^{-1}$, using instead~\eqref{e.infty.p}, we get the desired result.
\end{proof}

\subsection{Proof of the Harnack inequality}

We are now ready to present the proof of Theorem~\ref{t.Harnack.intro}. In fact, we prove the following statement containing a more explicit estimate.

\begin{theorem}[Harnack inequality]
  \label{t.Harnack}
  Let $s,t \in (0,1)$ with~$s+t<1$.
  There exists a constant $C(d)<\infty$ such that, for every positive and bounded~$u \in \mathcal{A}(\cu_0)$, we have for $\sigma\coloneqq 1-s-t > 0$
  \begin{equation}
    \sup_{\frac{1}{8}\cu_0} u
    \leq
    \exp\Bigl(
    C^{\frac{1}{\sigma(1-t)}} (\tfrac{C}{\sigma})^{\nicefrac{C}{\sigma}} t^{-1}
    \Theta_{s,t}^{\nicefrac{1}{2}}(\cu_0\,;\baconpow) \Bigr)
    \inf_{\frac{1}{8}\cu_0} u
    \,.
    \label{e.Harnack}
  \end{equation}
\end{theorem}
\begin{proof}
  Let~$p_\ast$ be as in Lemma~\ref{l.crossover} and~$\kappa \coloneqq \frac{d}{d-2(1-t)}$.
  Then we have, for~$\sigma\coloneqq 1-s-t > 0$,
  \begin{equation*}
    \|u^{p_\ast} \|_{\underline{L}^{1}(\frac{1}{4}\cu_0)}
    \|u^{-p_\ast}\|_{\underline{L}^{1}(\frac{1}{4}\cu_0)}
    \leq
    \exp\bigl( C^{\frac{1}{\sigma(1-t)}} (\tfrac{C}{\sigma})^{\nicefrac{C}{\sigma}} \bigr)
    \,.
  \end{equation*}
  An application of Proposition~\ref{p.moser.iteration.all.p} yields
  \begin{equation*}
    \|u^{-p_\ast}\|_{{L}^{\infty}(\frac18 \cu_0)}
    \leq
    \Bigl( C^{\frac{1}{1-t}}
    \bigl(1 + (\sigma t)^{-1 }|p_\ast| \Theta_{s,t}^{\nicefrac{1}{2}}(\cu_0\,;\baconpow ) \bigr)
    \Bigr)^{\nicefrac{d}{\sigma}} \|u^{-p_\ast}\|_{{L}^{1}(\frac14 \cu_0)}
    \,,
  \end{equation*}
  and applying~\eqref{e.infty.p} with~$p = \sqrt{\kappa}$ and then~\eqref{e.small.p} with~$p = p_\ast$ and~$q = \frac{1}{\sqrt{\kappa}}$ gives
  \begin{equation*}
    \|u^{p_\ast}\|_{{L}^{\infty}(\frac18 \cu_0)}
    \leq
    \Bigl( C^{\frac{1}{1-t}}
    \bigl(1 + (\sigma t)^{-1 } \Theta_{s,t}^{\nicefrac{1}{2}}(\cu_0\,;\baconpow ) \bigr)
    \Bigr)^{\nicefrac{d}{\sigma}} \|u^{p_{*}}\|_{{L}^{1}(\frac14 \cu_0)}
    \,.
  \end{equation*}
  Combining the above three displays yields
  \begin{equation*}
    \|u^{-p_\ast}\|_{{L}^{\infty}(\frac14 \cu_0)}
    \|u^{p_\ast}\|_{{L}^{\infty}(\frac14 \cu_0)}
    \leq
    \Bigl( C^{\frac{1}{1-t}} \bigl(1 + (\sigma t)^{-1 } \Theta_{s,t}^{\nicefrac{1}{2}}(\cu_0\,;\baconpow ) \bigr)  \Bigr)^{\!\nicefrac{d}{\sigma}}
    \exp\bigl( C^{\frac{1}{\sigma(1-t)}} (\tfrac{C}{\sigma})^{\nicefrac{C}{\sigma}} \bigr)
    \,,
  \end{equation*}
  which yields~\eqref{e.Harnack}.
\end{proof}

\begin{remark}[Optimality of the Harnack constant]
  \label{rem.optimal}
  To see that~$\exp(C\Theta^{\nicefrac12})$ is the optimal constant for the Harnack inequality in terms of dependence on the ellipticity ratio~$\Theta$, even for constant coefficient equations, observe that~$u(x_1,x_2) \coloneqq \exp( \Lambda^{\nicefrac12} x_1) \cos (x_2)$ is a smooth solution of
  \begin{equation}
    -\nabla \cdot \a_0 \nabla u = 0 \quad \mbox{in} \ \R^2\,,
    \quad \mbox{where} \quad
    \a_0 \coloneqq e_1\otimes e_1 + \Lambda e_2\otimes e_2
    \,,
  \end{equation}
  and that~$u$ is positive in the ball~$B_{\nicefrac \pi 2}$ and satisfies~$\sup_{B_r} u = u(0) \exp\big( \Lambda^{\frac12} r\big)$ for each~$r\in (0,\nicefrac \pi 2]$.
\end{remark}

\section{Examples}
\label{app.examples}

\subsection{Simple example}
We next record a simple example showing that the Besov condition~\eqref{e.Sobolev.ellipticity.condition} allows coefficient fields which are merely in~$L^1$ but not
in~$L^{1+\delta}$ for any~$\delta>0$. In the following example, we construct a coefficient field which does not belong to any~$L^p$-space for~$p>1$, but is nevertheless coarse-grained elliptic. For simplicity, in all our examples we assume the uniform lower bound~$\a \geq \Id$. 

\begin{proposition}
\label{p.example.L1.Besov.notLp1}
Let $s\in(1/2,1)$. Then there exists a nonnegative scalar coefficient field
$a \in L^1(\cu_0)$, such that
\begin{equation*}
a \not\in L^{1+\delta}(\cu_0)
\ \ \text{for every } \delta>0\qquad \mbox{and}\qquad \Lambda_{s}(\cu_{0};\bacon)<\infty.
\end{equation*}
\end{proposition}

\begin{proof}
We first construct a one-dimensional example on $(0,1)$ and then extend it to $\cu_0$ by
dependence on a single coordinate. Fix $s\in(1/2,1)$ and choose a parameter $\alpha\in(0,1)$.
For each integer $k\ge1$, set
\begin{equation*}
A_k \coloneqq 3^{k^2}
\quad\text{and}\quad
\ell_k \coloneqq 3^{-k}\,\frac{3^{\alpha k}}{A_k}.
\end{equation*}
Since $A_k> 3^{\alpha k}$, we have $\ell_k< 3^{-k}$ for every $k$. Let $J_k\subset(0,3^{-k})$ be defined as $J_{k}:=(0,\ell_{k})$, and set
\begin{equation*}
f(x) \coloneqq \sum_{k=1}^\infty A_k \, \mathbf{1}_{J_k}(x)
\qquad \text{for } x\in(0,1).
\end{equation*}

\smallskip

\noindent \emph{Step 1: $f\in L^1(0,1)$ but $f\notin L^p(0,1)$ for every $p>1$.}
For the $L^1$ norm we compute
\begin{equation*}
\int_0^1 f(x)\,dx
= \sum_{k=1}^\infty A_k |J_k|
= \sum_{k=1}^\infty A_k \, 3^{-k}\frac{3^{\alpha k}}{A_k}
= \sum_{k=1}^\infty 3^{-(1-\alpha)k} < \infty
\end{equation*}
since $\alpha<1$. For $p>1$, we obtain
\begin{equation*}
\int_0^1 |f(x)|^p\,dx
\ge  \sum_{k=1}^\infty A_k^p |J_k|
= \sum_{k=1}^\infty A_k^{p-1} 3^{-k}3^{\alpha k}
= \sum_{k=1}^\infty 3^{(p-1)k^2} \, 3^{-(1-\alpha)k}.
\end{equation*}
The exponent $(p-1)k^2-(1-\alpha)k$ tends to $+\infty$ for every fixed $p>1$, so the terms do not
tend to zero and the series diverges. Thus, $f\notin L^p(0,1)$ for every $p>1$.

\smallskip

\noindent \emph{Step 2: coarse-grained ellipticity.} For $k\ge1$, let $M_k$ denote the supremum of local averages of $f$ on intervals of length $3^{-k}$:
\begin{equation}
M_k \coloneqq \sup\biggl\{ \fint_I |f| \,dx : I\subset(0,1),\, |I|=3^{-k}\biggr\}.
\label{mk}
\end{equation}
We claim that $M_k \le c(\alpha,s) 3^{\frac{k(1+2s)}{2}}$. This suffices, as it is enough to bound
\begin{flalign}\label{imply1}
\sum_{k=1}^{\infty}3^{-sk}M_{k}^{\nicefrac{1}{2}} \ \stackrel{\eqref{Lam}}{\Longrightarrow} \ \Lambda_{s}(\cu_{0};\bacon)<\infty.
\end{flalign}
To estimate $M_k$, let us first observe that the general competitor interval in definition~\eqref{mk} is $I_{a}:=(a,a+3^{-k})$ for some $a\in (0,1-3^{-k})$, and that, as $|I_{a}|=3^{-k}$, the maximal average is achieved on $I_{0}$. In fact, for fixed $j\in \mathbb{N}$, we have
$$
|I_{a}\cap J_{j}|=\begin{cases}
    \displaystyle
    \ 0\quad &\mbox{if} \ \ a\ge \ell_{j}\vspace{1.5mm}\\
    \displaystyle
    \ \min\bigl\{\ell_{j},a+3^{-k}\bigr\}-a\quad &\mbox{if} \ \ a<\ell_{j}
\end{cases}\le |I_{0}\cap J_{j}|=\min\bigl\{\ell_{j},3^{-k}\bigr\},
$$
thus 
$$
\fint_{I_{a}}|f|\,dx=3^{k}\sum_{j=1}^{\infty}A_{j}|I_{a}\cap J_{j}|\le 3^{k}\sum_{j=1}^{\infty}A_{j}|I_{0}\cap J_{j}|=\fint_{I_{0}}|f|\,dx.
$$
Next, introduce the sets of indices
$$
\begin{cases}
    \displaystyle

    \ B_{1;k}:=\left\{j\in \mathbb{N}\colon 1\le j\le k-1\mbox{ and }\ell_{j}\le 3^{-k}\right\}\vspace{1.5mm}\\
    \displaystyle
    \ B_{2;k}:=\left\{j\in \mathbb{N}\colon 1\le j\le k-1\mbox{ and }\ell_{j}> 3^{-k} \right\}\vspace{1.5mm}\\
    \displaystyle
    \ B_{3;k}:=\left\{j\in \mathbb{N}\colon j\ge k\right\},
\end{cases}
$$
and notice that 
\begin{flalign}\label{b2k}
\begin{cases}
\displaystyle
    \ B_{1;k}\subseteq\biggl\{1,\cdots,\Bigl\lfloor\frac{\alpha-1+\sqrt{(1-\alpha)^{2}+4k}}{2} \Bigr\rfloor+1\biggr\}=:\tilde{B}_{1;k}\vspace{1.5mm}\\
    \displaystyle
    \ B_{2;k}\subseteq\biggl\{ \Bigl\lfloor\frac{\alpha-1+\sqrt{(1-\alpha)^{2}+4k}}{2}\Bigr\rfloor-1,\cdots,k-1\biggr\}=:\tilde{B}_{2;k}.
\end{cases}
\end{flalign}
We then split
\begin{eqnarray*}
\fint_{I_{0}}|f|\,dx&=&3^{k}\sum_{j=1}^{\infty}A_{j}|I_{0}\cap J_{j}|\nonumber \\
&\le&3^{k}\sum_{B_{1;k}}A_{j}|I_{0}\cap J_{j}|+3^{k}\sum_{B_{2;k}}A_{j}|I_{0}\cap J_{j}|+3^{k}\sum_{B_{3;k}}A_{j}|I_{0}\cap J_{j}|\nonumber \\
&=&S_{1;k}+S_{2;k}+S_{3;k}
\end{eqnarray*}
and bound
$$
S_{1;k}\le 3^{k}\sum_{\tilde{B}_{1;k}}3^{-j(1-\alpha)}\le 3^{k}\Bigl(\frac{3}{1-3^{-(1-\alpha)}} \Bigr).
$$
Similarly, we control
$$
S_{2;k}\le 3^{k}\sum_{\tilde{B}_{2;k}}3^{-j(1-\alpha)}\le \Bigl(\frac{3^{1-\alpha}}{1-3^{-(1-\alpha)^{2}}}\Bigr)3^{k-(1-\alpha)\sqrt{k}}
$$
and
$$
S_{3;k}\le 3^{k}\sum_{j\ge k}3^{-j(1-\alpha)}\le \frac{3^{\alpha k}}{1-3^{-(1-\alpha)}},
$$
so overall $M_{k}\le c(\alpha)3^{k}$. Plugging this information into~\eqref{imply1} and using~$s > \nicefrac12$, we obtain
$$
\sum_{k=1}^{\infty}3^{-sk}M_{k}^{\nicefrac{1}{2}}\le \sum_{k=1}^{\infty}3^{-k(s-\nicefrac{1}{2})}< \infty.
$$

\noindent \emph{Step 3: extension to $\cu_0$.}
Define a scalar coefficient field on $\cu_0=(-\tfrac12,\tfrac12)^d$ by
\begin{equation*}
a(x_1,\dots,x_d) \coloneqq f(|x_1|).
\end{equation*}
By Fubini, $a\in L^1(\cu_0)$ and $a\notin L^p(\cu_0)$ for every $p>1$ if and only if the same
holds for $f$ on $(0,1)$, so the conclusions of Step~1 transfer directly. Likewise, the
characterization of $B^{-2s}_{\infty,1}(\cu_0)$ in terms of scale-discounted averages on cubes
reduces to the one-dimensional estimates for $f$, and hence $a\in B^{-2s}_{\infty,1}(\cu_0)$ by
Step~2. This completes the proof.
\end{proof}

\subsection{Fractal example}
\label{ss.singular.measure}

We next discuss a second class of examples, based on mollified singular measures supported on
fractal sets. In particular, we show that the Besov condition~\eqref{e.Sobolev.ellipticity.condition} is stable under mollification of measures supported on sets
of Hausdorff dimension strictly larger than $d-2$, while the classical $L^p$ bounds deteriorate
along the approximating sequence.

\begin{proposition}[Fractal examples]
\label{p.example.fractal}
Let $d\ge2$ and $\alpha\in(d-2,d)$. Suppose that $F\Subset\cu_0$ is a compact
$\alpha$-Ahlfors regular set, in the sense that there exists a Borel probability measure $\mu_F$
supported on $F$ and a constant $C_0\ge1$ such that
\begin{equation}
C_0^{-1} r^{\alpha}
\;\le\;
\mu_F\bigl(B(x,r)\bigr)
\;\le\;
C_0 r^{\alpha}
\qquad\text{for all }x\in F,\ r\in(0,1).
\label{e.Frostman}
\end{equation}
Let $\eta\in C_c^\infty(\Rd)$ be a standard nonnegative mollifier with $\int_{\Rd}\eta=1$ and
$\supp\eta\subseteq B_1(0)$, and set $\eta_\delta(x)\coloneqq \delta^{-d}\eta(x/\delta)$ for
$\delta\in(0,1)$. Define the mollified densities
\begin{equation*}
\mu_F^\delta(x) \coloneqq (\mu_F * \eta_\delta)(x)
= \int_{\cu_0} \eta_\delta(x-y)\,d\mu_F(y),
\qquad x\in\cu_0,
\end{equation*}
and let $a_\delta \coloneqq 1 + \mu_F^\delta$. Then the following hold.
\begin{enumerate}
\item[\textup{(i)}] For every $\delta\in(0,\delta_{0})$, with $\delta_{0}:=\min\{1,\dist(F,\partial \cu_{0})\}/(4d)$,
\[
a_\delta \in L^1(\cu_0)
\qquad\text{and}\qquad
\int_{\cu_0} a_\delta(x)\,dx = 1 + \mu_F(\cu_0) = 2.
\]
\item[\textup{(ii)}] For every $p>1$,
\begin{equation*}
\lim_{\delta \to 0} 
\|\mu_F^\delta\|_{L^p(\cu_0)} = \infty
\,.
\end{equation*}

\item[\textup{(iii)}] Let $s\in(0,1)$ satisfy
\begin{equation}
2s > d-\alpha.
\label{e.fractal.s.condition}
\end{equation}
Then there exists a constant $C=C(d,\alpha,C_0,s,\eta)<\infty$, such that
\begin{equation*}
\sup_{\delta\in (0,1)}\Lambda_{s}(\cu_{0},\bacon) \le C.
\end{equation*}
In particular, the family $(a_\delta)_{\delta\in(0,1)}$ is coarse-grained elliptic for all $s\in  \left(\nicefrac{(d-\alpha)}{2},1\right)$ with constants independent of $\delta$, while the $L^p$-norms of $a_\delta$ diverge for every $p>1$.
\end{enumerate}
\end{proposition}

\begin{proof}
\emph{Step~1: $L^1$-normalization.}
Since $\eta_\delta$ has unit mass and $\mu_F$ is a probability measure, we have for every
$\delta\in(0,1)$
\[
\int_{\cu_0} \mu_F^\delta(x)\,dx
= \int_{\cu_0}\int_{\cu_{0}} \eta_\delta(x-y)\,dx\,d\mu_F(y)
= \int_{\cu_0} d\mu_F(y) = 1.
\]
Thus, $\|\mu_F^\delta\|_{L^1(\cu_0)}=1$ and $\|a_\delta\|_{L^1(\cu_0)}=2$ for all $\delta\in (0,\delta_{0})$.

\smallskip

\noindent \emph{Step~2: pointwise upper bound and local averages.}
From the upper Ahlfors regularity in~\eqref{e.Frostman} and the support of $\eta_\delta$, we obtain
a scale-dependent pointwise bound. Indeed, for any $x\in\cu_0$,
\begin{align*}
\mu_F^\delta(x)
&= \delta^{-d}\int_{\cu_0} \eta\!\left(\frac{x-y}{\delta}\right)\,d\mu_F(y)
\le C\,\delta^{-d}\mu_F\bigl(B(x,\delta)\bigr)
\le C_1\,\delta^{\alpha-d},
\end{align*}
for some constant $C_1=C_1(d,\alpha,C_0,\eta)$. Let $Q$ be a triadic cube contained in $\cu_0$ with side length $r\in(0,1)$ and center $x_Q$. Using Fubini and the bound on $\eta_\delta$, we estimate
\begin{align*}
\fint_Q \mu_F^\delta
= \frac{1}{|Q|}\int_Q \int_{\cu_0} \eta_\delta(x-y)\,d\mu_F(y)\,dx 
\leq \frac{C}{|Q|}\int_{\cu_0} \mathbf{1}_{\{ \dist(y,Q)\le\delta\}}\,d\mu_F(y)
= \frac{C(d)\,\mu_F(Q^{+\delta})}{|Q|}
\,,
\end{align*}
where $Q^{+\delta} \coloneqq \{y\in\cu_0 : \dist(y,Q)\le\delta\}$.
If $r\ge\delta$, then $Q^{+\delta}$ is contained in a ball of radius $2r$ centered at $x_Q$, and
therefore, by~\eqref{e.Frostman},
\[
\mu_F(Q^{+\delta})
\le \mu_F\bigl(B(x_Q,2r)\bigr)
\le C_0 (2r)^{\alpha}
\]
and hence
\begin{equation}
\fint_Q \mu_F^\delta
\le C\,\frac{r^{\alpha}}{r^d}
= C(d,C_{0},\alpha) r^{\alpha-d}
\qquad\text{whenever }r\ge\delta.
\label{e.fractal.avg.large}
\end{equation}
If $r<\delta$, we simply combine the pointwise bound with the trivial estimate
\[
\fint_Q \mu_F^\delta \le \sup_{x\in Q}\mu_F^\delta(x)\le C_1 \delta^{\alpha-d}.
\]
Summarizing, there exists $C_2=C_{2}(d,C_{0},\alpha,\eta)$ such that for every triadic cube $Q\subseteq\cu_0$ of side length
$r$,
\begin{equation}
\fint_Q \mu_F^\delta
\;\le\;
\begin{cases}
C_2\,r^{\alpha-d}, & r\ge\delta, \\
C_2\,\delta^{\alpha-d}, & r<\delta.
\end{cases}
\label{e.fractal.avg.two.regimes}
\end{equation}

\smallskip

\noindent \emph{Step~3: blow-up of $L^p$ norms.}
Let $p>1$ be fixed. Choose $\delta\in(0,\delta_{0})$ and let $\{x_i\}_{i=1}^N\subseteq F$ be a maximal $\delta$-separated set, that
is, $|x_i-x_j|\ge\delta$ for $i\neq j$ and $F\subseteq\bigcup_{i=1}^N B(x_i,\delta)$.
The balls $B(x_i,\nicefrac\delta2)$ are pairwise disjoint, and by the Ahlfors regularity
in~\eqref{e.Frostman},
\[
\begin{cases}
\displaystyle
\ 1 = \mu_F(F)
\ge \sum_{i=1}^N \mu_F\bigl(B(x_i,\tfrac{\delta}{2})\bigr)
\ge C_0^{-1} \Bigl(\frac{\delta}{2}\Bigr)^{\alpha} N\vspace{1.5mm}\\
\displaystyle
\ 1=\mu_{F}(F)\le \sum_{i=1}^{N}\mu_{F}(B(x_{i},\delta))\le C_{0}N\delta^{\alpha},
\end{cases}
\]
so that $N \approx_{\alpha,C_{0}}\delta^{-\alpha}$. Choose the mollifier $\eta$ so that $\eta\ge c_\eta>0$ on $B_{\nicefrac12}(0)$. If $x\in B(x_i,\nicefrac\delta8)$,
then $B(x_{i},\nicefrac{\delta}{8})\subseteq B(x,\nicefrac\delta4)\subseteq B(x_i,\nicefrac\delta2)$ and therefore
\begin{eqnarray*}
\mu_F^\delta(x)
&\ge& \delta^{-d}\int_{B(x,\nicefrac\delta4)} \eta\Bigl(\frac{x-y}{\delta}\Bigr)\,d\mu_F(y)\nonumber \\
&\ge& \delta^{-d}\int_{B(x_{i},\nicefrac\delta8)} \eta\Bigl(\frac{x-y}{\delta}\Bigr)\,d\mu_F(y)\ge c_\eta \delta^{-d}\,\mu_F\bigl(B(x_{i},\tfrac{\delta}{8})\bigr)
\ge C\delta^{\alpha-d},
\end{eqnarray*}
for some $C=C(\alpha,C_0,\eta)>0$. The balls $B(x_i,\nicefrac\delta8)$ are disjoint and all contained in
$\cu_0$ since $\delta<\delta_0$, so we obtain
\begin{equation*}
\int_{\cu_0} |\mu_F^\delta(x)|^p\,dx
\geq \sum_{i=1}^N \int_{B(x_i,\nicefrac\delta8)} |\mu_F^\delta(x)|^p\,dx 
\geq N c^p \delta^{p(\alpha-d)}\,|B(0,\tfrac{\delta}{8})|
\ge c(d,\alpha,C_{0},\eta,p) N\,\delta^{p(\alpha-d)+d}
\,.
\end{equation*}
\noindent Using $N\gtrsim_{C_{0},\alpha}\delta^{-\alpha}$ and simplifying yields
\[
\int_{\cu_0} |\mu_F^\delta|^p
\;\ge\; c(d,\alpha,C_{0},\eta,p)\delta^{(d-\alpha)(1-p)}.
\]
Since $d-\alpha>0$ and $p>1$, the exponent $(d-\alpha)(1-p)$ is negative, so the right-hand side
tends to $+\infty$ as $\delta\downarrow0$. Hence, $\|\mu_F^\delta\|_{L^p(\cu_0)}\to\infty$ for every
$p>1$.

\smallskip

\noindent \emph{Step~4: coarse-grained ellipticity.} Fix any $s \in \left(\nicefrac{(d-\alpha)}{2},1\right)$. For each $k\ge0$, let $\mathcal C_k$ denote the family of triadic cubes $Q\subseteq\cu_0$ of side
length $r_k\coloneqq3^{-k}$, and define
\[
M_k(\delta) \coloneqq
\max_{Q\in\mathcal C_k} \fint_Q |\mu_F^\delta|.
\]
By~\eqref{e.fractal.avg.two.regimes}, there exists $C_3=C_{3}(d,C_{0},\alpha,\eta)$ such that
\begin{equation*}
M_k(\delta)
\le
\begin{cases}
C_3\,r_k^{\alpha-d} = C_3\,3^{(d-\alpha)k}, & r_k\ge\delta, \\
C_3\,\delta^{\alpha-d}, & r_k<\delta.
\end{cases}
\end{equation*}
Let $k_0=k_0(\delta)$ be the integer such that $3^{-(k_0+1)}<\delta\le 3^{-k_0}$.
Then $r_k\ge\delta$ for $k\le k_0$ and $r_k<\delta$ for $k>k_0$, and we have by~\eqref{Lam}
\begin{equation*}
\Lambda_{s}(\cu_{0},\bacon)
\leq 
\sum_{k=0}^{k_0} 3^{-s k} M_k(\delta)^{\nicefrac{1}{2}}
+ \sum_{k>k_0} 3^{-s k} M_k(\delta)^{\nicefrac{1}{2}}
\eqqcolon S_1(\delta)+S_2(\delta)
\,.
\end{equation*}
For the first sum, using $M_k(\delta)\le C_3 3^{(d-\alpha)k}$ for $k\le k_0$ gives
\begin{equation*}
S_1(\delta)
\leq 
C_3^{\nicefrac{1}{2}}\sum_{k=0}^{k_0} 3^{-k\left(s-\frac{d-\alpha}{2}\right)}.
\,.
\end{equation*}
By the condition~\eqref{e.fractal.s.condition}, the exponent $d-\alpha-2s$ is negative, and
therefore $\sum_{k\ge0}3^{-k\left(s-\nicefrac{(d-\alpha)}{2}\right)}<\infty$. In particular, $S_1(\delta)$ is bounded uniformly in $\delta$. For the second sum, using $M_k(\delta)\le C_3\delta^{\alpha-d}$ for $k>k_0$ yields
\[
S_2(\delta)
\le C_3^{\nicefrac{1}{2}} \delta^{\frac{\alpha-d}{2}}\sum_{k>k_0}3^{-s k}
\le C_{3}^{\nicefrac{1}{2}} \delta^{\frac{\alpha-d}{2}} 3^{-s k_0}
\sum_{j\ge1}3^{-s j}
\le C(C_{3},s) \delta^{\frac{\alpha-d}{2}}\,3^{-s k_0}.
\]
Since $3^{-k_0}\simeq\delta$, we have $3^{-s k_0}\simeq \delta^{s}$, and thus
\[
S_2(\delta) \lesssim \delta^{\frac{\alpha-d}{2}+s}.
\]
Again by~\eqref{e.fractal.s.condition}, the exponent $\nicefrac{(\alpha-d)}{2}+s$ is positive, so
$\delta^{\frac{\alpha-d+2s}{2}}\le1$ for $\delta\in(0,1)$. Hence, $S_2(\delta)$ is also bounded uniformly in
$\delta$, and we conclude that
\begin{equation*}
\sup_{0<\delta<1}
\Lambda_{s}(\cu_{0},\bacon) <\infty
\,.
\end{equation*}
The contribution of the constant $1$ in $a_\delta=1+\mu_F^\delta$ to $\Lambda_{s}$ is
\[
\sum_{k\ge0} 3^{-s k}
\max_{Q\in\mathcal C_k} \left(\fint_Q 1
\right)^{\nicefrac{1}{2}}= \sum_{k\ge0} 3^{-s k} <\infty,
\]
so $a_\delta$ is coarse-grained elliptic uniformly in $\delta$ as well. This
completes the proof.
\end{proof}

\subsection{Multifractal random measures}
\label{ss.GMC}

We finally record a third family of examples, based on Gaussian multiplicative chaos. These show that the Besov condition~\eqref{e.Sobolev.ellipticity.condition} applies to regularizations
of highly intermittent log-normal weights which are uniformly in $L^1$ but not in $L^p$ for any
$p>1$. In the Proposition below, we work with mollified versions of the measure and prove that $\a_\delta = 1+\mu_\gamma^\delta \in \Besov{\infty}{1}{-2s}$ which implies that $\Lambda_s(\cu_{-1})$ is uniformly bounded in $\delta$, furthermore the addition of $1$ ensures that $\Theta_{s,t}(\cu_{-1})$ is also uniformly bounded in $\delta$ for all $t\ge0$ and hence $\a_\delta$ is coarse-grained elliptic in $\cu_{-1}$.

\begin{proposition}[Gaussian multiplicative chaos]\label{p.GMC.example}
  Let $d\ge2$ and let $X$ be a centered log-correlated Gaussian field on $\cu_0$ with covariance
  \begin{equation*}
    \E\big[X(x)X(y)\big]
    =
    \log\frac{1}{|x-y|}
    +
    g(x,y),
  \end{equation*}
  where $g$ is bounded and smooth.
  For $\gamma\in(0,\sqrt{2d})$, let $\mu_\gamma$ denote the associated (subcritical) Gaussian multiplicative chaos on $\cu_0$, that is,
  \begin{equation*}
    \mu_\gamma(dx) = e^{\diamond\,\gamma X(x)}\,dx,
  \end{equation*}
  where $e^{\diamond\,\gamma X(x)}$ denotes the Wick exponential of $\gamma X(x)$.
  Let $\eta\in C_c^\infty(\Rd)$ be a standard nonnegative mollifier with $\int_{\Rd}\eta=1$, set $\eta_\delta(x)\coloneqq \delta^{-d}\eta(x/\delta)$, and define
  \begin{equation*}
    \mu_\gamma^\delta \coloneqq \mu_\gamma * \eta_\delta,
    \qquad
    \a_\delta \coloneqq 1 + \mu_\gamma^\delta,
    \qquad
    \delta\in(0,1).
  \end{equation*}
  There exists a constant $c_d \in (0,\sqrt{2d})$ depending only on $d$ such that,
  almost surely, for every $\gamma\in(0,c_d)$, the following hold:
  \begin{enumerate}
    \item\label{i.uniform.L1}
          \begin{equation*}
            \sup_{\delta\in(0,1)} \|\a_\delta\|_{L^1(\cu_0)}
            \le 1+\mu_\gamma(\cu_0)
            <\infty.
          \end{equation*}

    \item\label{i.blowup.Lp}
          For every $p>1$,
          \begin{equation*}
            \lim_{\delta\downarrow0}\|\a_\delta\|_{L^p(\cu_{-1})}=\infty.
          \end{equation*}

    \item\label{i.negative.Besov}
          For every $s\in(\gamma/c_d,1)$,
          \begin{equation*}
            \sup_{\delta\in(0,1)} \|\a_\delta\|_{B^{-2s}_{\infty,1}(\cu_{-1})}<\infty.
          \end{equation*}
  \end{enumerate}
\end{proposition}

\begin{proof}
  Existence of $\mu_\gamma$ for log-correlated fields as above, and the subcritical range $\gamma^2<2d$, are classical; see, e.g.,~\cite{Kahane1985,RhodesVargas2014,Shamov2016}. Fix a realization of the field and the chaos and suppress this dependence in the notation.

  \smallskip

  \emph{Step 1: uniform $L^1$ bounds.}
  By definition of convolution and since $\eta_\delta\ge0$ has unit mass,
  \begin{equation*}
    \int_{\cu_0} \mu_\gamma^\delta(x)\,dx
    =
    \int_{\cu_0}\int_{\cu_0} \eta_\delta(x-y)\,dx\,d\mu_\gamma(y)
    \le \mu_\gamma(\cu_0),
    \qquad \delta\in(0,1).
  \end{equation*}
  Hence, $\|\a_\delta\|_{L^1(\cu_0)} \le |\cu_0| + \mu_\gamma(\cu_0)$, proving~\ref{i.uniform.L1}.

  \smallskip

  \emph{Step 2: blow-up in $L^p$.}
  It is well known that for every $\gamma>0$ the measure $\mu_\gamma$ is almost surely singular with respect to Lebesgue measure; see, for instance,~\cite[Theorem~2.6]{RhodesVargas2014}.
  On the other hand, if $\nu$ is a finite Borel measure on $\cu_{-1}$ and $(\eta_\delta)_{\delta>0}$ is an approximate identity, then for any $p>1$,
  \begin{equation*}
    \sup_{\delta\in(0,1)}\|\nu*\eta_\delta\|_{L^p(\cu_{-1})}<\infty
    \quad\Longrightarrow\quad
    \nu \ll dx
    \ \text{and}\
    \frac{d\nu}{dx}\in L^p(\cu_{-1}),
  \end{equation*}
  so the uniform $L^p$-boundedness of mollifications forces absolute continuity.
  Applying this with $\nu=\mu_\gamma$ and using $\mu_\gamma\perp dx$, we conclude that for every $p>1$,
  \begin{equation*}
    \|\mu_\gamma^\delta\|_{L^p(\cu_{-1})}\to\infty
    \qquad\text{as }\delta\downarrow0.
  \end{equation*}
  Since $\a_\delta=1+\mu_\gamma^\delta$, this implies~\ref{i.blowup.Lp}.

  \smallskip

  \emph{Step 3: negative Besov regularity.}
  The Besov regularity of Gaussian multiplicative chaos has been analyzed in detail
  in~\cite{JunnilaSaksmanViitasaari2019}. Specializing~\cite[Theorem~4]{JunnilaSaksmanViitasaari2019} yields that for every $p\in(1,\infty)$ satisfying $\gamma^2<\nicefrac{2d}{p^2}$ and every $q\ge1$,
  \begin{equation*}
    \mu_\gamma \in B^\sigma_{p,q,\mathrm{loc}}(\cu_0)
    \qquad\text{for all }\sigma<-\frac{(p-1)\gamma^2}{2}.
  \end{equation*}
  Restricting to $\cu_{-1}$ and using standard Besov embeddings (see, e.g.,~\cite{Triebel1983}),
  \begin{equation*}
    B^\sigma_{p,q}(\cu_{-1})
    \hookrightarrow
    B^{\sigma-\nicefrac{d}{p}}_{\infty,q}(\cu_{-1}),
    \qquad 1<p<\infty,\ q\ge1,
  \end{equation*}
  we infer that for each $p>1$ with $\gamma^2<\nicefrac{2d}{p^2}$,
  \begin{equation}\label{e.GMC.Besov.regularity}
    \mu_\gamma \in B^\tau_{\infty,1}(\cu_{-1})
    \qquad\text{for all }\tau<-\frac{(p-1)\gamma^2}{2}-\frac{d}{p}.
  \end{equation}
  Choose $p^\ast\coloneqq \sqrt d/\gamma$. Then $\gamma^2<\nicefrac{2d}{(p^\ast)^2}=2\gamma^2$ and
  \begin{equation*}
    -\frac{(p^\ast-1)\gamma^2}{2}-\frac{d}{p^\ast}
    =
    \frac{\gamma(\gamma-3\sqrt d)}{2}.
  \end{equation*}
  Therefore, whenever $\gamma$ is small enough $p^\ast > 1$, and we can pick $\varepsilon>0$ such that
  \begin{equation*}
    2s \coloneqq \frac{\gamma(3\sqrt d-\gamma)}{2}-\varepsilon \in (0,1),
  \end{equation*}
  and then~\eqref{e.GMC.Besov.regularity} implies $\mu_\gamma \in B^{-2s}_{\infty,1}(\cu_{-1})$.

  Since convolution with a smooth compactly supported kernel is bounded on $B^{-2s}_{\infty,1}$,
  we have
  \begin{equation*}
    \sup_{\delta\in(0,1)} \|\mu_\gamma^\delta\|_{B^{-2s}_{\infty,1}(\cu_{-1})}<\infty,
  \end{equation*}
  and hence the same estimate for $\a_\delta=1+\mu_\gamma^\delta$. Renaming constants yields~\ref{i.negative.Besov}.
\end{proof}

\appendix

\section{Auxiliary lemmas}

The proof of the coarse-grained Caccioppoli inequality (Proposition~\ref{p.caccioppoli}) makes use of the following well-known iteration lemma. We also need it in Lemma~\ref{l.bombieri}, below. For the convenience of the reader, we present a complete proof. 

\begin{lemma}\label{l.iteration}
  Suppose that $A,\xi > 0$ and $\varrho:[\nicefrac{1}{2},T) \to [0,\infty)$  satisfy
\begin{equation*}\sup_{t \in [\nicefrac{1}{2},T)} (T-t)^\xi \varrho(t) < \infty
\end{equation*}
and, for every $\nicefrac{1}{2} \leq s < t \leq T$,
\begin{equation*}\varrho(s) \leq \frac{1}{2} \varrho(t) + \frac{A}{(t-s)^\xi}
\,.
\end{equation*}
Then there exists a universal constant $C > 1$ such that
\begin{equation*}\varrho(s) \leq C (2\xi)^\xi (T-s)^{-\xi} A
\,.
\end{equation*}
\end{lemma}
\begin{proof}
  Denote $M \coloneqq \sup_{t \in [\nicefrac{1}{2},T)} (T-t)^\xi \varrho(t) < \infty$. Fix $s \in [\nicefrac{1}{2},T)$ and take $\delta \in (0,\nicefrac12)$ to be fixed and let $t \coloneqq s+ \delta(T-s)$. Then $t-s = \delta(T-s)$ and $T-t = (T-s)(1-\delta)$. Thus,
\begin{equation*}
(T-s)^\xi \varrho(s) \leq \frac{1}{2} 
\Bigl ( \frac{T-s}{T-t}\Bigr )^\xi(T-t)^\xi \varrho(t) + \frac{(T-s)^\xi}{(t-s)^\xi} A = \frac{1}{2} \frac{1}{(1-\delta)^\xi} (T-t)^\xi \varrho(t) + \delta^{-\xi} A
\end{equation*}
  This yields
\begin{equation*}(T-s)^\xi \varrho(s) \leq \frac{1}{2} \frac{1}{(1-\delta)^\xi} (T-t)^\xi \varrho(t) + \delta^{-\xi} A
\,.
\end{equation*}
  Taking the supremum over $s \in [\nicefrac{1}{2},T)$, we have
\begin{equation*}
  M \leq \frac{\delta^{-\xi}}{1-\nicefrac{1}{2}(1-\delta)^{-\xi}} A
\,.
\end{equation*}
  In particular since $\varrho(s') \leq (T-s')^{-\xi} M$, for any $s' \in [\nicefrac{1}{2},T)$ we have
\begin{equation*}
  \varrho(s') \leq \frac{\delta^{-\xi}}{1-\nicefrac{1}{2}(1-\delta)^{-\xi}} (T-s')^{-\xi} A.
\end{equation*}
  From the above we see that any $\delta \in (0,1-2^{-\nicefrac{1}{\xi}})$ works. Optimizing over $\delta$ gives the desired result.
\end{proof}

As in~\cite{BG}, our proof of the crossover lemma in Section~\ref{ss.crossover} is based on the following lemma.  

\begin{lemma}[Bombieri's lemma]
\label{l.bombieri}
Let~$q \in (0,\infty)$,~$\xi \in (0,\infty]$ and $A_1, A_2 \geq 1$. Suppose that~$v \in L^q(\cu_0)$ satisfies~$\inf_{\cu_0} v > 0$,
\begin{equation*}
\left \| \log v \right\|_{L^1(\cu_0)} \leq A_1,
\end{equation*}
 and, for every~$p \in (0,q)$ and~$\nicefrac{1}{2} \leq r < s \leq 1$,
\begin{equation*}\|v\|_{L^q(r\cu_0)} \leq \Bigl ( \frac{A_2}{(s-r)^\xi} \Bigr )^{\nicefrac{1}{p}} \|v\|_{L^{p}(s\cu_0)}.
\end{equation*}
Then there exists a constant $C < \infty$ such that
\begin{equation*}\|v\|_{L^q(\frac{1}{2} \cu_0)} \leq \exp \left(C (1+\nicefrac 1q) (C\xi)^{6\xi} A_1 A_2^6\right).
\end{equation*}
\end{lemma}
\begin{proof}
  Fix $r,s$ such that $\nicefrac{1}{2} \leq r < s \leq 1$ and define $f(s) \coloneqq \log\|v\|_{L^q(s\cu_0)}$. We then have
  \begin{align*}
\|v\|_{L^p(s \cu_0)}
\leq&
\|\indc_{\{ \log v \leq f(s)/3\}} v\|_{L^p(s \cu_0)} + \|\indc_{\{ \log v > f(s)/3\}} v\|_{L^p(s \cu_0)}
\\
& \leq
    \exp\left( \frac13 f(s)  \right) + \|v\|_{L^q(s \cu_0)} \biggl ( \int_{\cu_0} \indc_{\{ \log v > f(s)/3\}} \, dx \biggr )^{\!\frac{1}{p} - \frac{1}{q}}
\\
& \leq
     \exp\Bigl( \frac13 f(s)  \Bigr) +  \exp( f(s) ) \Bigl ( \frac{3 A_1}{f(s)}  \Bigr )^{\frac{1}{p} - \frac{1}{q}}
  \end{align*}
  If $f(s) > 3 A_1$, we can choose, by the intermediate value theorem, an exponent~$p$ such that
\begin{equation*}\frac{ \log \left ( \frac{f(s)}{ 3 A_1}  \right )}{f(s)} =    
    \left (\frac{1}{p} - \frac{1}{q} \right )^{-1} \,,
\end{equation*}
  which, combined with the above, gives us
\begin{equation*}\|v\|_{L^p(s \cu_0)} \leq \exp\Bigl( \frac13 f(s)  \Bigr)  + 1 \leq 2\exp\Bigl( \frac13 f(s)  \Bigr) \,.
\end{equation*}
  Using the assumption on the reverse Hölder's inequality, we then get
\begin{align*}
f(r) \leq  \frac{1}{p} \log\left ( \frac{A_2}{(s-r)^\xi} \right ) + \log \|v\|_{L^p(s \cu_0)} 
\leq f(s) \Biggl ( \frac{\log ( \frac{A_2}{(s-r)^\xi}  )}{\log  ( \frac{f(s)}{3A_1}   )}  + \frac13  \Biggr ) + \log 2
    +\frac{1}{q} \log \left ( \frac{A_2}{(s-r)^\xi} \right )
    \,.
  \end{align*}
  Next, if we have
\begin{equation*}3A_1 < \frac{3 A_1 A_2^6}{(s-r)^{6\xi}} \leq f(s)
    \implies 
\log \Bigl ( \frac{f(s)}{3A_1}  \Bigr ) \geq 6 \log\Bigl ( \frac{A_2}{(s-r)^\xi} \Bigr )
\,,
\end{equation*}
then we obtain
\begin{equation*}f(r) \leq \frac{1}{2} f(s) + \log 2 +\frac{1}{q} \log \Bigl ( \frac{A_2}{(s-r)^\xi} \Bigr ) \,.
\end{equation*}
Otherwise, we get
\begin{equation*}f(r) \leq f(s) \leq \frac{3 A_1 A_2^2}{(s-r)^{2\xi}}\,.
\end{equation*}
  Putting these two cases together we have, for every $\nicefrac{1}{2} \leq r < s \leq 1$,
\begin{equation*}f(r) \leq \frac{1}{2} f(s) + \frac{6A_1 A_2^6}{(s-r)^{6\xi}}
    +\frac{1}{q} \log \Bigl ( \frac{A_2}{(s-r)^\xi} \Bigr )
     \leq
     \frac12 f(s) + (1+\nicefrac 1q)\frac{6A_1 A_2^6}{(s-r)^{6\xi}} \,.
\end{equation*}
An application of Lemma~\ref{l.iteration} gives us
\begin{equation*}f(\nicefrac{1}{2}) \leq C(1+\nicefrac 1q)  (12\xi)^{6\xi} A_1 A_2^6 \,.
\end{equation*}
  Exponentiating this gives the desired result.
\end{proof}
\section{Sobolev spaces and test functions} \label{app.sobolev}

For the basic properties of the weighted Sobolev space~$H^1_\a(\cu_0)$, recall Section~\ref{ss.sobolev.spaces}. We now prove the following chain rule, which is used in the proof of the Caccioppoli inequality (Proposition~\ref{p.caccioppoli}).

\begin{lemma}
  \label{l.chain.rule}
  Consider a smooth function~$F:\R \to \R$ such that~$F'\in L^\infty(\R)$. Then, for every~$u \in H_{\a}^1(\cu_0)$, we have~$F(u) \in H_{\a}^1(\cu_0)$ and the weak chain rule is valid:
\begin{equation*}
\nabla (F(u)) = F'(u) \nabla u
\,.
\end{equation*}
\end{lemma}
\begin{proof}
  Since $H^1_\a(\cu_0)$ is defined as the closure of smooth functions we can take a sequence of smooth functions $u_m$ converging to $u$ in $H^1_\a(\cu_0)$.
  Since $F'$ is bounded, we have trivially
\begin{equation*}\| F(u_m) - F(u) \|_{\underline{L}^2(\cu_0)} \leq \|F'\|_{L^\infty(\R)} \| u_m - u \|_{\underline{L}^2(\cu_0)}.
\end{equation*}
  Moreover,
\begin{equation*}F'(u_m) \nabla u_m - F'(u) \nabla u = F'(u_m) (\nabla u_m - \nabla u) + (F'(u_m) - F'(u)) \nabla u.
\end{equation*}
  For the first term
\begin{equation*}\fint_{\cu_0} |F'(u_m) \a^{\nicefrac{1}{2}}(\nabla u_m - \nabla u)|^2 \leq \|F'\|_{L^\infty(\R)}^2 \fint_{\cu_0} |\a^{\nicefrac{1}{2}}(\nabla u_m - \nabla u)|^2 \to 0.
\end{equation*}
  For the second term,
\begin{equation*}\fint_{\cu_0} |(F'(u_m) - F'(u)) \a^{\nicefrac{1}{2}} \nabla u|^2 \leq C \|F'\|_{L^\infty(\R)}^2 \fint_{\cu_0} |\a^{\nicefrac{1}{2}} \nabla u|^2 < \infty.
\end{equation*}
  And since $u_m \to u$ in $L^2(\cu_0)$, we can extract a subsequence such that $u_m \to u$ pointwise a.e.~in $\cu_0$. By the continuity of $F'$, we have $F'(u_m) \to F'(u)$ pointwise a.e.~in $\cu_0$. Therefore, by the dominated convergence theorem,
\begin{equation*}\fint_{\cu_0} |(F'(u_m) - F'(u)) \a^{\nicefrac{1}{2}} \nabla u|^2 \to 0.
\end{equation*}
  Finally, we need to identify the weak derivative of $F(u)$. By abuse of notation we let $\varphi$ denote the vector-valued function with all components in $C^\infty_c(\cu_0)$, then
\begin{equation*}\fint_{\cu_0} F(u) \nabla \cdot \varphi = \lim_{m \to \infty} \fint_{\cu_0} F(u_m) \nabla \cdot \varphi = -\lim_{m \to \infty} \fint_{\cu_0} F'(u_m) \nabla u_m \cdot \varphi
\end{equation*}
  and
  \begin{align*}
    \biggl| \fint_{\cu_0} F'(u_m) \nabla u_m \cdot \varphi - \fint_{\cu_0} F'(u) \nabla u \cdot \varphi \biggr| 
 & =
    \biggl| \fint_{\cu_0} \a^{\nicefrac{1}{2}}(F'(u_m) \nabla u_m - F'(u) \nabla u) \cdot \a^{-\nicefrac{1}{2}}\varphi
    \biggr| 
\\
 & \leq
\| \a^{\nicefrac{1}{2}}(F'(u_m) \nabla u_m - F'(u) \nabla u) \|_{\underline{L}^2(\cu_0)} \| \a^{-\nicefrac{1}{2}}\varphi \|_{\underline{L}^2(\cu_0)} \to 0,
  \end{align*}
  since $\varphi$ is bounded and $\a^{-\nicefrac{1}{2}} \in L^2(\cu_0)$. Thus, the weak derivative of $F(u)$ is $F'(u) \nabla u$.
\end{proof}

From the above we can cover the case of $F(u) = |u|$ by the classical approximation with $F_\varepsilon(u) = \sqrt{u^2 + \varepsilon^2}-\varepsilon$, which is smooth and converges to $|u|$ as $\varepsilon \to 0$, also satisfying the assumptions of Lemma~\ref{l.chain.rule}. Now by dominated convergence and the classical interpretation of the weak derivative of $|u|$ we have that $|u| \in H_\a^1(\cu_0)$ and, consequently,~$u_+, u_- \in H_\a^1(\cu_0)$.

Next, we prove that the testing procedure used in the proof of the Caccioppoli inequality (Proposition~\ref{p.caccioppoli}) is valid under the assumptions on the fluxes given by the finiteness of the Besov norms $\Theta_{s,1-s}(\cu_0;\baconsol)$ and $\Theta_{s,1-s}(\cu_0;\baconsub)$.

\begin{lemma} \label{l.product.rule}
  Let $u \in H^1_\a(\cu_0)$.
  \begin{enumerate}[label=\textup{(\alph*)}]
    \item If $u$ is a weak solution of $-\nabla \cdot (\a \nabla u) = 0$ in $\cu_0$ and
    $\Theta_{s,1-s}(\cu_0;\baconsol) < \infty$ for some $s \in (0,1)$, then for every $g \in H^1_\a(\cu_0)$ and $\varphi \in C^\infty_c(\cu_0)$ we have
    \begin{equation} \label{e.product.rule}
      0
      = \fint_{\cu_0} \varphi \,\a \nabla u \cdot \nabla g
      + \fint_{\cu_0} g \,\a \nabla u \cdot \nabla \varphi\,.
    \end{equation}
    \item If $u$ is a weak nonnegative subsolution of $-\nabla \cdot (\a \nabla u) = 0$ in $\cu_0$ and
    $\Theta_{s,1-s}(\cu_0;\baconsub) < \infty$ for some $s \in (0,1)$, then for every $0 \le g \in H^1_\a(\cu_0)$ and $0 \le \varphi \in C^\infty_c(\cu_0)$ we have
    \begin{equation} \label{e.product.rule.sub}
      0 
      \ge \fint_{\cu_0} \varphi \,\a \nabla u \cdot \nabla g
      + \fint_{\cu_0} g \,\a \nabla u \cdot \nabla \varphi\,.
    \end{equation}
  \end{enumerate}
\end{lemma}

\begin{proof}
  \phantom{.}
  \smallskip

  \emph{Step 1: Proof for bounded~$g$.}  
  Let $g \in H^1_\a(\cu_0) \cap L^\infty(\cu_0)$. Since $H^1_\a(\cu_0)$ is defined as the closure of smooth functions, there exists a sequence $g_m \in C^\infty(\cu_0) \cap L^\infty(\cu_0)$ with
  $g_m \to g$ in $H^1_\a(\cu_0)$. For each $m$, we have the following for $v_m = g-g_m$
  \begin{equation*}
    \fint_{\cu_0} \a \nabla (v_m \varphi) \cdot \nabla (v_m \varphi)
    =
    \fint_{\cu_0} \varphi^2\, \a \nabla v_m \cdot \nabla v_m + 2 \fint_{\cu_0} v_m \varphi\, \a \nabla v_m \cdot \nabla \varphi + \fint_{\cu_0} v_m^2\, \a \nabla \varphi \cdot \nabla \varphi\,.
  \end{equation*}
  Since $g_m$ and $g$ are bounded, by using Cauchy-Schwarz, the fact that $\a \in L^1(\cu_0)$ and the convergence in $H^1_\a(\cu_0)$ of $g_m$ to $g$, we can pass to the limit as $m \to \infty$ to obtain that $g_m \varphi \to g\varphi \in H^1_{\a,0}(\cu_0)$.

  Secondly, for each $m$, we may test the weak formulation with
  $\psi = g_m \varphi \in C^\infty_c(\cu_0)$ to obtain
  \begin{equation*}
    0
    = \fint_{\cu_0} \a \nabla u \cdot \nabla (g_m \varphi)
    = \fint_{\cu_0} \varphi \,\a \nabla u \cdot \nabla g_m
    + \fint_{\cu_0} g_m \,\a \nabla u \cdot \nabla \varphi \,.
  \end{equation*}
  By Cauchy–Schwarz, the first term converges directly, since $g_m \to g$ in $H^1_\a(\cu_0)$.
  It remains to pass to the limit in
  \begin{equation*}
    \fint_{\cu_0} g_m \,\a \nabla u \cdot \nabla \varphi\,.
  \end{equation*}

  Since $\Theta_{s,1-s}(\cu_0;\baconsol) < \infty$,
  Lemma~\ref{l.besov.fluxes.easy} implies that
  $\a \nabla u \in \BesovDualSum{2}{1}{-s}(\cu_0)$ and
  $\nabla u \in \BesovDualSum{2}{1}{s-1}(\cu_0)$.
  Moreover, by Lemma~\ref{l.besov.poincare} and Lemma~\ref{l.besov.fluxes.easy},
  we have $(g-g_m)\nabla \varphi \in \Besov{2}{\infty}{s}(\cu_0)$.
  Using the duality between Besov spaces, the relation~\eqref{e.dual.norms.relation},
  and Lemmas~\ref{l.besov.poincare} and~\ref{l.besov.fluxes.easy}, we obtain a constant
  $C < \infty$ independent of $m$ such that
  \begin{equation} \label{e.besov.convergence}
    \begin{aligned}
      \biggl| \fint_{\cu_0} (g-g_m) \,\a \nabla u \cdot \nabla \varphi \biggr|
       & \leq
      C \,\|(g-g_m) \nabla \varphi\|_{\Besov{2}{\infty}{s}(\cu_0)}
      \,\| \a \nabla u \|_{\BesovDualSum{2}{1}{-s}(\cu_0)} \\
       & \leq
      C\,\|g-g_m \|_{\underline{H}^1_\a(\cu_0)}
      \,\|\a^{\nicefrac{1}{2}} \nabla u \|_{\underline{L}^2(\cu_0)} \;\longrightarrow\; 0
    \end{aligned}
  \end{equation}
  as $m\to\infty$. Hence,
  \begin{equation*}
    \fint_{\cu_0} g_m \,\a \nabla u \cdot \nabla \varphi
    \;\longrightarrow\;
    \fint_{\cu_0} g \,\a \nabla u \cdot \nabla \varphi,
  \end{equation*}
  and passing to the limit in the identity above yields~\eqref{e.product.rule}.

  \smallskip

  \emph{Step 2: General case.}
  
  Let $g \in H^1_\a(\cu_0)$ and $\varphi \in C^\infty_c(\cu_0)$. Construct $g_k := \max(\min(g,k),-k)$, and note that $g_k \in H^1_\a(\cu_0) \cap L^\infty(\cu_0)$ by the remark after Lemma~\ref{l.chain.rule}. Note that $g_k \to g$ in $H^1_\a(\cu_0)$ as $k \to \infty$ by the dominated convergence theorem. Applying~\eqref{e.product.rule} to $g_k$ gives
  \begin{equation*}
    0
    = \fint_{\cu_0} \varphi \,\a \nabla u \cdot \nabla g_k
    + \fint_{\cu_0} g_k \,\a \nabla u \cdot \nabla \varphi\,.
  \end{equation*}
  We note that the proof is identical to the one in Step 1, since~\eqref{e.besov.convergence} only depends on the fact that $g_k \to g$ in $H^1_\a(\cu_0)$ as $k \to \infty$. Thus, we can pass to the limit as $k \to \infty$ to obtain~\eqref{e.product.rule} for general $g \in H^1_\a(\cu_0)$.

  \smallskip

  \emph{Step 3: Subsolutions.}
  We start with $0 \leq g \in H^1_\a(\cu_0) \cap L^\infty(\cu_0)$ and $0 \leq \varphi \in C^\infty_c(\cu_0)$ as in Step 1. Assuming that $u$ is a weak non-negative subsolution of $-\nabla \cdot (\a \nabla u) = 0$ in $\cu_0$, and that $\Theta_{s,1-s}(\cu_0;\baconsub\,) < \infty$ for some $s \in (0,1)$, we can use the same argument as in Step 1, where $0 \leq g_m \in C^\infty(\cu_0) \cap L^\infty(\cu_0)$ with $g_m \to g \in H^1_\a(\cu_0)$ to pass to the limit as $m \to \infty$ in
  \begin{equation*}
    0 \geq \fint_{\cu_0} \a \nabla u \cdot \nabla (g_m \varphi) = \fint_{\cu_0} \varphi \,\a \nabla u \cdot \nabla g_m + \fint_{\cu_0} g_m \,\a \nabla u \cdot \nabla \varphi,
  \end{equation*}
  which is possible since, $g_m \varphi \to g \varphi \in H^1_{\a,0}(\cu_0)$ and $g_m \to g$ in $H^1_\a(\cu_0)$, giving
  \begin{equation} \label{e.subsolution.test.bounded}
    0 \geq \fint_{\cu_0} \a \nabla u \cdot \nabla (g \varphi) = \fint_{\cu_0} \varphi \,\a \nabla u \cdot \nabla g + \fint_{\cu_0} g \,\a \nabla u \cdot \nabla \varphi\,.
  \end{equation}
  The general case now follows in the same way as in Step 2.
\end{proof}

\begingroup
\small 

\subsubsection*{\bf Acknowledgments}
S.~A. and T.~K. acknowledge support from the European Research Council (ERC) under the European Union's Horizon Europe research and innovation programme, grant agreement number 101200828. T.~K. was supported by the Academy of Finland. C.~D. and G.~M. acknowledge support from the European Research Council (ERC) under the European Union's Horizon Europe research and innovation programme, grant agreement number 101220121, and from the University of Parma through the action "Bando di Ateneo 2024 per la ricerca".

{
\bibliographystyle{alpha}
\bibliography{refs}
}
\endgroup

\end{document}